\documentclass[a4paper,oneside,12pt, reqno]{amsart}

\usepackage{hyperref}
\hypersetup{colorlinks=true, urlcolor=red, linkcolor=blue, citecolor=purple, linktoc=none}

\usepackage{mathrsfs}
\usepackage{amsfonts,amssymb,amsmath,amsthm}
\usepackage{url}
\usepackage{enumerate}

\usepackage{bbm}
\usepackage{float}
\usepackage{pstricks}
\usepackage{amscd}
\usepackage{amsmath}
\usepackage{amsxtra}
\usepackage[T1]{fontenc}
\usepackage[utf8]{inputenc}
\usepackage{lipsum}
\usepackage{tikz}
\usepackage{tikz-cd}
\usetikzlibrary{arrows}
\usetikzlibrary{calc}
\usepackage{hyperref}
\usepackage{eucal}
\usepackage{soul}

\newcommand{\HT}{H^2(\triangle_H)}
\newcommand{\HD}{H^{2}(\mathbb{D}^2)}
\newcommand{\HDO}{H^{2}(\mathbb{D})}
\newcommand{\TRH}{\triangle_H}
\newcommand{\HTO}{H^2_+(\TRH)}

\theoremstyle{plain}
\newtheorem{theorem}{Theorem}[section]

\newtheorem{lemma}[theorem]{Lemma}
\newtheorem{proposition}[theorem]{Proposition}
\newtheorem{corollary}[theorem]{Corollary}

\theoremstyle{definition}
\newtheorem{definition}[theorem]{Definition}

\newtheorem{example}[theorem]{Example}

\theoremstyle{remark}
\newtheorem{remark}[theorem]{Remark}

\newtheorem{case-new}{Case}

\numberwithin{equation}{section}

\usepackage{etoolbox} 

\makeatletter
\@namedef{subjclassname@2020}{%
  \textup{2020} Mathematics Subject Classification}
\makeatother

\usepackage{csquotes}
\usepackage{mathrsfs}
\usepackage{epsfig}
\usepackage[active]{srcltx}

\newcommand{\ncom}{\newcommand}
\ncom{\bq}{\begin{equation}}
\ncom{\eq}{\end{equation}}
\ncom{\beqn}{\begin{eqnarray*}}
\ncom{\eeqn}{\end{eqnarray*}}
\ncom{\beq}{\begin{eqnarray}}
\ncom{\eeq}{\end{eqnarray}}
\ncom{\nno}{\nonumber}
\ncom{\rar}{\rightarrow}
\ncom{\Rar}{\Rightarrow}
\ncom{\noin}{\noindent}
\ncom{\bc}{\begin{centre}}
\ncom{\ec}{\end{centre}}
\ncom{\sz}{\scriptsize}
\ncom{\rf}{\ref}
\ncom{\sgm}{\sigma}
\ncom{\Sgm}{\Sigma}
\ncom{\dt}{\delta}
\ncom{\Dt}{Delta}
\ncom{\lmd}{\lambda}
\ncom{\Lmd}{\Lambda}
\ncom{\eps}{\epsilon}
\ncom{\pcc}{\stackrel{P}{>}}
\ncom{\dist}{{\rm\,dist}}
\ncom{\sspan}{{\rm\,span}}
\ncom{\im}{{\rm Im\,}}
\ncom{\sgn}{{\rm sgn\,}}
\ncom{\ba}{\begin{array}}
\ncom{\ea}{\end{array}}
\ncom{\eop}{\hfill{{\rule{2.5mm}{2.5mm}}}}
\ncom{\eoex}{\hfill{$\diamond$}}
\ncom{\eoe}{\hfill{{\rule{1.5mm}{1.5mm}}}}
\ncom{\eof}{\hfill{{\rule{1.5mm}{1.5mm}}}}
\ncom{\hone}{\mbox{\hspace{1em}}}
\ncom{\htwo}{\mbox{\hspace{2em}}}
\ncom{\hthree}{\mbox{\hspace{3em}}}
\ncom{\hfour}{\mbox{\hspace{4em}}}
\ncom{\hsev}{\mbox{\hspace{7em}}}
\ncom{\vone}{\vskip 2ex}
\ncom{\vtwo}{\vskip 4ex}
\ncom{\vonee}{\vskip 1.5ex}
\ncom{\vthree}{\vskip 6ex}
\ncom{\vfour}{\vspace*{8ex}}
\ncom{\norm}{\|\;\;\|}
\ncom{\integ}[4]{\int_{#1}^{#2}\,{#3}\,d{#4}}
\ncom{\inp}[2]{\langle{#1},\,{#2} \rangle}
\ncom{\Inp}[2]{\Langle{#1},\,{#2} \Langle}
\ncom{\vspan}[1]{{{\rm\,span}\#1 \}}}
\ncom{\dm}[1]{\displaystyle {#1}}

\ncom{\Dm}{\textnormal{d}m}

\ncom{\cred}{\color{red}}
\ncom{\cpl}{\color{purple}}
\ncom{\cblue}{\color{blue}}



\begin{document}
\title[doubly commutativity and essential normality of modules]{On modules of the Hardy space of Hartogs triangle}
\author[Arup Chattopadhyay]{Arup Chattopadhyay}
\address{Department of Mathematics, Indian Institute of Technology Guwahati, Guwahati, 781039, India}
\email{arupchatt@iitg.ac.in, 2003arupchattopadhyay@gmail.com}

\author[Saikat Giri]{Saikat Giri}
\address{Department of Mathematics, Indian Institute of Technology Guwahati, Guwahati, 781039, India}
\email{saikat.giri@iitg.ac.in, saikatgiri90@gmail.com}

\author[Shubham Jain]{Shubham Jain}
\address{Department of Mathematics, Indian Institute of Technology Guwahati, Guwahati, 781039, India}
\email{shubjainiitg@iitg.ac.in, shubhamjain4343@gmail.com}

\begin{abstract} 
In this paper, we investigate the structure of doubly commuting submodules and quotient modules of the Hardy space $H^2(\triangle_H)$ over the Hartogs triangle. We establish a complete classification of doubly commuting submodules. In addition, we characterize all doubly commuting quotient modules of the form $(\theta_1(z/w)\theta_2(w)H^2(\triangle_H))^\perp$, where $\theta_1$ and $\theta_2$ are inner functions on the unit disc. This is achieved by introducing the concept of $\varphi$-doubly commuting quotient modules on the Hardy space $H^2(\mathbb D^2).$ We further explore the essential normality and doubly commutativity of quotient modules of the form $(pH^2(\triangle_H))^\perp$ under some mild assumptions on $p$, where $p$ is a polynomial in two variables.
\end{abstract}

\maketitle
\section{Introduction}
The domain Hartogs triangle, denoted by $\TRH,$ is defined as 
\beqn
\TRH:=\{(z, w)\in\mathbb{C}^2:|z|<|w|<1\}.
\eeqn
$\TRH$ is a bounded and pseudoconvex domain. However, unlike the unit polydisc or the unit ball, it is neither monomially convex nor polynomially convex (see \cite{S2015}). Consider the biholomorphism $\varphi$ from $\TRH$ onto $\mathbb{D}\times \mathbb{D}^*$ given by the map 
\beq \label{phi-map} \varphi(z,w)=\left(\varphi_1(z, w), \varphi_2(z, w)\right)=\left(\frac{z}{w}, w \right), \quad (z, w) \in \TRH,~\mathbb{D}^*=\mathbb{D}\backslash\{0\}.
\eeq
Note that the Jacobian $J_{\varphi}$ of $\varphi$ is given by $J_{\varphi}(z,w)=\frac{1}{w}$ for all $(z, w) \in \TRH$.
Furthermore, the inverse map $\varphi^{-1}$ and the Jacobian $J_{\varphi^{-1}}$ of $\varphi^{-1}$ are given by
\allowdisplaybreaks
\beq \label{phi inverse}
\varphi^{-1}(z,w)=\Big(zw, w\Big),\quad J_{\varphi^{-1}}(z,w)=w, \quad (z, w)\in\mathbb{D}\times\mathbb{D}^*. 
\eeq

Recall from \cite[Section 3]{Mo21} (see also \cite[Section 7]{CJP2023}) that the Hardy space $\HT$ on $\TRH$ is defined as
\begin{align*}
\HT=\Big\{f\in\textnormal{Hol}(\TRH): \sup_{s, t \in (0, 1)}\frac{1}{4\pi^2} \int_0^{2\pi}\int_0^{2\pi}  \!\Big|f\Big(st e^{i \theta}, te^{i \gamma}\Big)\Big|^2st^2 d\theta d \gamma<\infty\Big\}.
\end{align*}
Also, recall that the Hardy space $H^2(\mathbb D^2)$ on the unit bidisc $\mathbb D^2$ is defined as 
\beqn
H^2(\mathbb D^2):=\Big\{f \in \textnormal{Hol}(\mathbb D^2):\sup_{t \in [0,1)}\int_{\mathbb T^2}|f(tz_1, tz_2)|^2\Dm(z) < \infty \Big \},
\eeqn
where $\Dm(z)$ is the normalized Lebesgue measure on $\mathbb T^2$. The space $H^2(\mathbb D^2)$ may be identified with the closed subspace $H^2(\mathbb T^2)$ of $L^2(\mathbb T^2)$ (\cite[Theorem 3.4.3]{Ru69}).

The distinguished boundary $\partial_{d}\triangle_H$ of $\TRH$ is $\mathbb{T}^2$ (\cite[Remark 2.2]{JP2024}). It is shown in \cite[Section 3]{Mo21} (see also \cite[Proposition 2.4]{JP2024}) that the Hardy space $\HT$ is isometrically isomorphic to a closed subspace $H^2(\partial_d \triangle_H)$ of $L^2(\mathbb T^2)$, given by 
\beqn H^2(\partial_{d}\triangle_H):=\Big\{ f(\zeta)=\sum_{\alpha \in \mathcal {I}}a_{\alpha}\zeta^{\alpha}:\sum_{\alpha \in \mathcal {I}}|a_{\alpha}|^2<\infty \Big \},
\eeqn
where $\zeta=(\zeta_{1},\zeta_{2})\in\mathbb{T}^{2}$, and $\mathcal{I}=\{\alpha=(\alpha_{1},\alpha_{2})\in \mathbb{Z}^2:\alpha_1\ge 0,~\alpha_{1}+\alpha_{2}+1 \ge 0\}$.

We now recall the following relation from \cite[Proposition 3.4]{Mo21} between the inner products of $H^2(\triangle_{H})$ and $H^{2}(\mathbb{D}^{2})$: for $f,g\in\HT$, we have
\begin{align}\label{Inner Product Relation}
\langle f,g\rangle_{\HT}=\langle \Psi(f),\Psi(g)\rangle_{\HD},
\end{align}
where the map $\Psi : H^2(\partial_d\triangle_H)\rar H^2(\mathbb T^2)$ is unitary and given by 
\begin{align}\label{Psi-map}
\Psi(f)=J_{\varphi^{-1}}\cdot f\circ \varphi^{-1}, \quad f \in H^2(\partial_d\triangle_H).
\end{align}
Note that $\Psi^{-1}: H^2(\mathbb T^2)\rar H^2(\partial_d\triangle_H)$ is given by
\begin{align*}
\Psi^{-1}(f)=J_{\varphi}\cdot f\circ \varphi, \quad f \in H^2(\mathbb T^2).
\end{align*}
Consider the closed subspace $H^2_+(\TRH)$ of $\HT$ defined as
\beq \label{hdinht}
H^2_+(\TRH) := \left\{ f \in \HT : f(z,w) = \sum_{i,j \ge 0} a_{i,j} z^i w^j \right\}.
\eeq
By Abel’s lemma \cite[Lemma 1.5.8]{S2005}, it follows that the map $R: \HD \to H^2_+(\TRH)$, given by $R(f) = f|_{\TRH}$, is unitary. Hence, we may identify $\HD$ as a subspace of $\HT$ via the map $R$.

We denote by $H^{\infty}(\Omega)$ the Banach algebra of all bounded holomorphic functions on $\Omega$, endowed with the sup norm, where $\Omega=\mathbb{D}^{2}$ or $\TRH$.

\begin{remark}\label{hd-inner-ht}
Note that $H^{\infty}(\mathbb{D}^{2})\subsetneq H^{\infty}(\TRH)$, and moreover $\phi\in H^{\infty}(\TRH)$ if and only if $\phi\circ\varphi^{-1}\in H^{\infty}(\mathbb{D}^{2})$ (see \cite[Remark 3.1]{JP2024}). Since $\varphi^{-1}(\mathbb{T}^{2})=\mathbb{T}^{2}$, therefore $\phi\in H^{\infty}(\TRH)$ is inner if and only if $\phi\circ\varphi^{-1}\in H^{\infty}(\mathbb{D}^{2})$ is inner (see \cite{JP2024}).\eof
\end{remark}
 
\subsection{Notations and conventions} 
Let $\mathbb{Z}$, $\mathbb{Z}_{+}$, and $\mathbb{N}$ denote the set of integers, nonnegative integers, and natural numbers, respectively. Let $\mathbb{D}$ and $\mathbb{T}$ denote the open unit disc and the unit circle in $\mathbb{C}$, respectively. For a bounded linear operator $T$ on a Hilbert space $\mathcal{H}$, we denote its range and kernel by $\operatorname{ran}(T)$ and $\operatorname{ker}(T)$, respectively. For a closed subspace $\mathcal{M}$ of a Hilbert space $\mathcal{H}$, the notation $\mathcal{H}\ominus\mathcal{M}$ denotes the orthogonal complement of $\mathcal{M}$ in $\mathcal{H}$. We denote by $\mathbb{C}[z, w]$ the ring of polynomials in the variables $z$ and $w$ with complex coefficients.

Let $\mathcal{H}$ be a Hilbert module \cite{DP1989} over $\mathbb{C}[z,w]$, and let $M_{z}$ and $M_{w}$ denote the module multiplication operators on $\mathcal{H}$. A closed subspace $\mathcal{S}$ (resp., $\mathcal{Q}$) of $\mathcal{H}$ is said to be a submodule (resp., quotient module) of $\mathcal{H}$ if $M_{\theta}(\mathcal{S})\subseteq\mathcal{S}$ (resp., $M_{\theta}^*(\mathcal{Q})\subseteq\mathcal{Q}$) for all $\theta\in\{z,w\}$.  A submodule $\mathcal{S}$ (or a quotient module $\mathcal{Q}$) is said to be \emph{nontrivial} if $\mathcal{S}\neq\{0\}$ (resp., $\mathcal{Q}\neq\{0\}$).

The module multiplication operators on a submodule $\mathcal{S}$ and a quotient module $\mathcal{Q}$ of a Hilbert module $\mathcal{H}$ are given by the restrictions $(S_{z},S_{w})$ and compressions $(Q_{z},Q_{w})$ of the module multiplication operators on $\mathcal{H}$, respectively. That is,
$$S_z=M_z|_{\mathcal{S}},\ \ S_w=M_w|_{\mathcal{S}}\quad\text{and}\quad Q_{z}=P_{\mathcal{Q}}M_{z}|_{\mathcal{Q}},\ \ Q_{w}=P_{\mathcal{Q}}M_{w}|_{\mathcal{Q}}.$$
Here and throughout, for any closed subspace $\mathcal{M}$ of a Hilbert space $\mathcal{H}$, we denote by $P_{\mathcal{M}}$ the orthogonal projection from $\mathcal{H}$ onto $\mathcal{M}$.

We further say that:
\begin{itemize}
\item A submodule $\mathcal{S}$ $(\mbox{or a quotient module }\mathcal{Q})$ is \emph{doubly commuting} if $[S_z^*, S_w] = 0$ $(\mbox{resp., } [Q_z^*, Q_w]=0)$;
\item A quotient module $\mathcal{Q}$ is \emph{essentially normal} if the commutator $[Q_{\phi}^*, Q_{\psi}]$ is compact for all $\phi, \psi \in \{z, w\}$.
\end{itemize}
Here $[\cdot,\cdot]$ denotes the commutator $[A,B]:=AB-BA$, for bounded linear operators $A,B$ on $\mathcal{H}$.

In this article, we primarily study two Hilbert modules, $\HD$ and $\HT$. The operators $M_z, M_w$ (resp., $M_1, M_2$) denote multiplication by $z$ and $w$ on $\HT$ (resp., $\HD$), respectively. It is shown in \cite[Proposition 8.3]{CJP2023} that
\begin{equation}\label{Diagram}
M_1M_2=\Psi M_z\Psi^{-1} \quad \mbox{ and }\quad M_2=\Psi M_w\Psi^{-1}.
\end{equation}

We use the notation $\mathcal{S}$ and $\mathfrak{S}$ to denote submodules of $\HT$ and $\HD$, respectively, with $\mathcal{Q}$ and $\mathfrak{Q}$ representing the corresponding quotient modules. The module multiplication operators on the submodules of $\HT$ and $\HD$ are denoted by $(S_z, S_w)$ and $(S_1, S_2)$, respectively, while the module multiplication operators on the corresponding quotient modules are denoted by $(Q_z, Q_w)$ and $(Q_1, Q_2)$.

\subsection{Main results:} The theory of invariant subspaces has remained a central area of investigation in operator theory for several decades. Classical results, most notably Beurling’s theorem on the invariant subspaces of the Hardy space $\HDO$ \cite{B1949}, have laid the foundation for extensive developments in both the one-variable and multivariable settings. In several complex variables, the invariant subspace problem becomes significantly more intricate due to the absence of a full analog of Beurling's theorem. Nonetheless, substantial progress has been made in specific cases, such as doubly commuting shifts; see, for instance, \cite{Ru69, M1988, GM1988} and references therein, and in the study of backward shift-invariant subspaces of $H^2(\mathbb{D}^n) $; see, for instance, \cite{DY2000, DY1998, INS2004, III2011, Sarkar14} and references therein. The essential normality of polynomially generated quotient modules of $H^2(\mathbb{D}^n)$ has been thoroughly studied; see, for example, \cite{DM1993, GuWa2007, GW2009, DaGoSa20}. These directions continue to inspire the structural analysis of submodules and quotient modules of Hilbert spaces of analytic functions. While the literature on invariant subspaces and essential normality is vast, we restrict our attention to the works most relevant to our present considerations.

This work is devoted to the classification of doubly commuting submodules and a partial classification of doubly commuting quotient modules, with an analysis of the essentially normal quotient modules of the form $(p\HT)^{\perp}$, where $p$ is a polynomial. We now state the main results of the paper in the order they are proved:
\vspace*{0.1cm}

\textbf{Section}~\ref{Sec2}: In \cite[Theorem 2]{M1988}, Mandrekar showed that a nontrivial submodule $\mathfrak{S}$ of $H^2(\mathbb{D}^2)$ is doubly commuting if and only if $\mathfrak{S}=q H^2(\mathbb D^2)$ for some inner function $q$ on $\mathbb{D}^2$. Motivated by this result, we prove the following in the context of $\HT.$

\begin{theorem}\label{Doubly-com-sub}
A nontrivial submodule $\mathcal{S}$  of $\HT$ is doubly commuting if and only if $\mathcal{S}=q\HTO$ for some unimodular function $q$ on $\triangle_H$.
\end{theorem}

As a consequence, we establish that the rank of a doubly commuting submodule of $\HT$ is $1$ (Corollary \ref{Rank}). Furthermore, we show in Proposition \ref{NOT-Doubly-com-sub} that a submodule of the form $q\HT$ is never doubly commuting, where $q$ is an inner function on $\TRH$.
\vspace*{0.1cm}

\textbf{Section}~\ref{Sec3}:  In \cite[Theorem 2.1]{INS2004}, Izuchi, Nakazi and Seto provided a characterization of doubly commuting quotient modules of $\HD$ (see also \cite[Corollary 3.3]{Sarkar14}). We prove the following in the context of $\HT$:

\begin{theorem} \label{Hartogscharasepa}
For inner functions $\theta_{1}$ and $\theta_{2}$ on $\mathbb{D},$  the quotient module $\HT\ominus \theta_1(z/w)\theta_2(w)\HT$ is doubly commuting if and only if one of the following holds:
\begin{itemize}
\item [$\rm{(i)}$] $\theta_{1}$ and $\theta_{2}$ are constant functions.
\item [$\rm{(ii)}$] $\theta_1$ is constant and $\theta_2(w)=c\dfrac{w-b}{1- \bar b w}$ for some $b \in \mathbb D$ and $c \in \mathbb T.$
\item [$\rm{(iii)}$] $\theta_2$ is constant and $\theta_1(z)=c_1z$ for some $c_1 \in \mathbb{T}.$ 
\item [$\rm{(iv)}$] $\theta_1(z)=c_1z$ and $\theta_2(w)=c_2w$ for some $c_1, c_2\in \mathbb{T}.$
\end{itemize}
\end{theorem}
To achieve this, we introduce the notion of $\varphi$-doubly commuting quotient modules of $\HD$ (see Definition \ref{phi-dc}). We begin by providing a complete characterization of $\varphi$-doubly commuting quotient modules within the class of doubly commuting quotient modules (see Theorem \ref{DC+varDC}). This with Corollary \ref{keycoro1},  Lemmas \ref{Keylemma1}-\ref{keylemma3} and Theorem \ref{maincharasepa} yields complete characterization of $\varphi$-doubly commuting quotient modules of the form $\HD\ominus \theta_1(z)\theta_2(w)\HD,$ where $\theta_1$ and $\theta_2$ are inner functions on $\mathbb D.$ (see Theorem \ref{finalcharasepa}). We end this section by providing a class $\varphi$-doubly commuting quotient modules where the inner function is not separating (see Proposition \ref{notsepatheta}).
\vspace*{0.1cm}

\textbf{Section} \ref{Sec4}: In \cite{DM1993}, Douglas and Misra established by a direct calculation that $\HD\ominus(z-w)^{2}\HD$ is essentially normal, while $\HD\ominus z^{2}\HD$ is not. Later, in \cite{GuWa2007}, Guo and Wang provided a characterization of the homogeneous polynomials $p\in\mathbb{C}[z,w]$ for which the quotient module $\HD\ominus p\HD$ is essentially normal (see also \cite{Clark}, \cite{DaGoSa20}, \cite{GW2009}, and the references theirin). In the following result, we have completely characterized polynomials $p\in\mathbb{C}[z,w]$ with some assumptions on $p$ for which the associated quotient module is not essentially normal: 

\begin{theorem}\label{Not essentially normal Theorem}
Let $p$ be a polynomial in the variables $z$ and $w$ represented as 
\beqn
p(z,w)=a z+ b w + cw^2 +w^3q(w) +z^2r(z)+\sum_{i,j \ge 1} a_{i,j} z^i w^j,
\eeqn
where $a, b, c$ are scalars, and $q(w)$ and $r(z)$ are single variable polynomials such that $\mathcal{Q}_p=\HT\ominus p\HT$ satisfies \eqref{Class-p}. Then $\mathcal{Q}_p$ is not essentially normal if and only if $a=b=0$.
\end{theorem}
To this end, we carry out a detailed analysis of the dimension of $ \mathcal{Q}_p \cap \mathfrak{F}_m$ for  $m \in \mathbb{Z}_+ $ and $p \in \mathbb C[z, w],$ where $ \mathfrak{F}_m $ denotes the span of the functions $\frac{1}{w} \left( \frac{z}{w} \right)^{m-j} w^{j}$ for $0 \le j \le m$ (see Subsection \ref{Dim analysis}). We conclude the section by giving a class of doubly commuting quotient modules in this setting (see Proposition \ref{dcpolycase}).

The paper includes many examples and counterexamples, along with several supporting results that may be of independent interest. This paper is nearly self-contained.

\section{Submodules: doubly commuting and properties}\label{Sec2}
This section focuses on the study of doubly commuting submodules of $\HT$ and explores some of their applications.
It is easy to see that $\HD$ is a doubly commuting submodule of itself (see, e.g., \cite[Theorem 2]{M1988}). In contrast, $\HT$ is not a doubly commuting submodule of itself. The following result shows that the Beurling-type submodules of $\HT$ are never doubly commuting.
\begin{proposition}\label{NOT-Doubly-com-sub}
For an inner function $q$ on $\triangle_H,$ $q \HT$ is never doubly commuting submodule.
\end{proposition}

\begin{proof}
Let $q$ be an inner function on $\TRH.$ Then $\mathcal{S}=q\HT$ is a submodule of $\HT.$ Since $q$ is inner, $M_q$ is an isometry on $\HT$ (see \cite[Corollary 4.4]{JP2024}) and hence $\left\{q\frac{1}{w}\left(\frac{z}{w}\right)^iw^j : i, j \ge 0\right\}$ is an orthonormal basis of $\mathcal S.$ Now for, any $g \in \HT,$
\begin{align*}
\left \langle S_w^*q\frac{1}{w}\left(\frac{z}{w}\right)^iw^{j}, qg \right \rangle &=\left \langle P_{\mathcal{S}}M_w^*q\frac{1}{w}\left(\frac{z}{w}\right)^iw^{j}, qg \right \rangle \\&=\left\langle q\frac{1}{w}\left(\frac{z}{w}\right)^iw^{j}, qwg \right \rangle=\left\langle \frac{1}{w}\left(\frac{z}{w}\right)^iw^{j}, wg \right \rangle.
\end{align*}
This yields that for all $i\ge0$, and $j\ge 1$, $S_w^*q\frac{1}{w}\left(\frac{z}{w}\right)^i =0$, and  $$S_w^*q\frac{1}{w}\left(\frac{z}{w}\right)^iw^j =q\frac{1}{w}\left(\frac{z}{w}\right)^iw^{j-1}.$$ 
Thus
$$S_w^*S_zq\frac{1}{w}\left(\frac{z}{w}\right)^i=S_w^*q\frac{1}{w}\left(\frac{z}{w}\right)^{i+1}w=q\frac{1}{w}\left(\frac{z}{w}\right)^{i+1},$$
whereas $S_zS_w^*q\frac{1}{w}\left(\frac{z}{w}\right)^i=0.$ Therefore, $\mathcal{S}$ is not doubly commuting.
\end{proof}

In the sequel, we say that a function $\phi:\TRH\to\mathbb{C}$ is unimodular if $|\phi|=1$ a.e. on $\partial_{d}(\TRH)=\mathbb{T}^{2}$. It is important to note that not every unimodular Hardy function on $\TRH$ is inner. For instance, the function $q \in \HT$ defined by $q(z,w) = \frac{1}{w}$ is unimodular but fails to be inner.

\begin{remark}\label{uni-inner}
For a unimodular $\phi\in\HT,$ it is easy to see that the function $\Psi(\phi)$ is inner on $\mathbb{D}^{2}$ (see \eqref{Psi-map} for $\Psi$).\eof
\end{remark}

It is easy to see that $\HTO$ is a doubly commuting submodule of $\HT$ (see \eqref{hdinht}). Theorem \ref{Doubly-com-sub} characterizes all doubly commuting submodules in terms of $\HTO.$ The proof of Theorem \ref{Doubly-com-sub} follows the approach outlined in \cite{M1988}. Before proceeding, we note the following elementary lemma.
\begin{lemma}\label{Mq-isometry}
For a unimodular function $q:\TRH\to\mathbb{C}$, the operator $M_{q}:\HTO\to\HT$ is an isometry.
\end{lemma}

\begin{proof}
Let $f,g\in\HTO$, then
\begin{align*}
\langle M_{q}f,M_{q}g\rangle_{\HT}&\overset{\eqref{Inner Product Relation}}{=}\langle \Psi(qf),\Psi(qg)\rangle_{\HD}\\
&=\langle \Psi(q)\cdot f\circ\varphi^{-1},\Psi(q)\cdot g\circ\varphi^{-1}\rangle_{\HD}\\
&=\langle f\circ\varphi^{-1},g\circ\varphi^{-1}\rangle_{\HD}\quad\mbox{(see Remark \ref{uni-inner})}\\
&=\langle M_{2}f\circ\varphi^{-1},M_{2}g\circ\varphi^{-1}\rangle_{\HD}\\
&=\langle \Psi(f),\Psi(g)\rangle_{\HD}\\
&=\langle f,g\rangle_{\HT}.
\end{align*}
This completes the proof.
\end{proof}
\begin{proof}[Proof of Theorem \ref{Doubly-com-sub}]
Let $\mathcal{S}=q\HTO$ for some unimodular function $q$ on $\triangle_H$. We claim that $S_z$ and $S_w$ are doubly commuting on $\mathcal{S}.$ Since $q$ is unimodular, by Lemma \ref{Mq-isometry}, $M_q$ is an isometry from $\HTO$ to $\HT.$ Hence $\{q z^i w^j:i,j\ge 0\}$ is an orthonormal basis of $\mathcal{S}.$ Now for any $g\in\HTO,$
$$\langle S_w^*qz^iw^j, qg \rangle=\langle M_w^*qz^iw^j, qg \rangle =\langle qz^iw^j, qwg \rangle =\langle z^iw^j, wg \rangle.$$
This implies that for all $i\ge0$,
$$S_{w}^*(qz^{i}w^{j})=\begin{cases}
0&\text{if}~j=0\\
qz^iw^{j-1}&\text{if}~j\ge1.
\end{cases}$$  
Now, for $i\ge0$ and $j\ge 1$, we have
$$S_w^*S_zqz^i=0=S_zS_w^*qz^i\quad\mbox{and}\quad S_w^*S_zqz^{i}w^{j}=qz^{i+1}w^{j-1}=S_zS_w^*qz^{i}w^{j}.$$
Therefore, $S_z$ and $S_w$ are doubly commuting on $\mathcal{S}$.
	
Let $S_{z}$ and $S_{w}$ be doubly commuting on the submodule $\mathcal{S}$. Then, in view of \cite[Theorem 1]{S1980}, we get
\begin{align}\label{Sub-Decomp}
\mathcal{S}=\bigoplus_{m,n=0}^{\infty} z^{m}w^{n}\left(\mathcal{W}_{1}\cap\mathcal{W}_{2}\right),
\end{align}
where $\mathcal{W}_{1}=\mathcal{S}\ominus z\mathcal{S}$ and $\mathcal{W}_{2}=\mathcal{S}\ominus w\mathcal{S}$. Again from \cite[Corollary 1 and Theorem 1]{S1980}, we further obtain $\mathcal{W}_{1}\cap\mathcal{W}_{2}\neq\{0\}$. We now claim that $\mathcal{W}_{1}\cap\mathcal{W}_{2}$ is one dimensional. Form the above argument, it suffices to show that the dimension of $\mathcal{W}_{1}\cap\mathcal{W}_{2}$ is at most one. Let $g_1,g_2\in\mathcal{W}_{1}\cap\mathcal{W}_{2}$, and define $\psi_1=g_1\circ \varphi^{-1}$ and $\psi_2=g_2\circ \varphi^{-1}$ (see \eqref{phi inverse} for $\varphi^{-1}$).

Since $S_z$ and $S_w$ are doubly commuting, by \cite[Lemma 6.2]{MaSa2019}, we have $S_{z}\mathcal{W}_{2}\subseteq\mathcal{W}_{2}$ and $S_{w}\mathcal{W}_{1}\subseteq\mathcal{W}_{1}$. This, together with \eqref{Inner Product Relation}, further implies
\begin{align*}
\langle \Psi(z^{m}g_{1}),\Psi(w^{n}g_{2})\rangle_{\HD}=0\quad\forall~m\ge0~\text{and}~n>0,
\end{align*}
which further reduces to
\begin{align}\label{DCS-1}
\int_{\mathbb{T}^{2}}z^mw^{m-n}\psi_1(z, w)\overline{\psi_2(z, w)}d\sigma(z, w)=0\quad\forall~m\ge0~\text{and}~n>0,
\end{align}
where $d\sigma(z, w)$ denotes the normalized Lebesgue measure on $\mathbb T^2.$
We also have 
\begin{align*}
\langle \Psi(w^{n}g_{1}),\Psi(z^{m}g_{2})\rangle_{\HD}=0\quad\forall~m\ge0~\text{and}~n>0,
\end{align*}
which is same as saying
\begin{align}\label{DCS-2}
\int_{\mathbb{T}^{2}}z^mw^{m+n}\psi_1(z, w)\overline{\psi_2(z, w)}d\sigma(z, w)=0\quad\forall~m\le0~\text{and}~n>0.
\end{align}
We also note that \eqref{DCS-1} and \eqref{DCS-2} hold for $m>0,n\ge0$ and $m<0,n\ge0$, respectively. Moreover, since $g_{2}\in\mathcal{W}_{2}$ and $z^{m}w^{n}g_{1}\in w\mathcal{S}$ for all $m\ge0$ and $n>0$, it follows that
\begin{align}\label{DCS-4}
\int_{\mathbb{T}^{2}}z^mw^{m+n}\psi_1(z, w)\overline{\psi_2(z, w)}d\sigma(z, w)=0\quad\forall~m\ge0~\text{and}~n>0.
\end{align}
Similary, since $g_{1}\in\mathcal{W}_{2}$ and $z^{m}w^{n}g_{2}\in w\mathcal{S}$ for all $m\ge0$ and $n>0$, we obtain
\begin{align}\label{DCS-5}
\int_{\mathbb{T}^{2}}z^mw^{m+n}\psi_1(z, w)\overline{\psi_2(z, w)}d\sigma(z, w)=0\quad\forall~m\le0~\text{and}~n<0.
\end{align}
Consequently, from \eqref{DCS-1}-\eqref{DCS-5}, we deduce that
\begin{align}\label{DCS-3}
\int_{\mathbb{T}^{2}}z^mw^{n}\psi_1(z, w)\overline{\psi_2(z, w)}d\sigma(z, w)=0\quad\forall~m,n\neq(0,0).
\end{align}
The later equation \eqref{DCS-3} implies that $\psi_1\overline{\psi_2}=c$ a.e. on $\mathbb{T}^2$. Now by arguing similarly to the proof of \cite[Theorem 2]{M1988}, we find that $\mathcal{W}_1\cap\mathcal{W}_2$ is one-dimensional. Let $q$ generate $\mathcal{W}_1\cap\mathcal{W}_2$, one can choose $q\circ\varphi^{-1}$ to be unimodular on $\mathbb T^2.$ Hence, $q\in\HT$ is also unimodular. Finally, \eqref{Sub-Decomp} gives the result. This completes the proof of the second implication.
\end{proof}
We now present an alternate proof of Theorem \ref{Doubly-com-sub}.
\begin{proof}[Alternate proof of Theorem \ref{Doubly-com-sub}]
Note that $S_z$ and $S_w$ are shifts on any submodule $\mathcal{S}$ of $\HT.$ Now the result follows from \cite[Corollary 4]{GM1988}.
\end{proof}

The following example shows that the unimodular function $q$ is not required to be inner in Theorem \ref{Doubly-com-sub}, in contrast to the classification of doubly commuting submodules of $\HD$, where $q$ must be inner (see \cite[Theorem 2]{M1988}).
\begin{example}
Consider the submodule $\mathcal{S}=\frac{1}{w}\HTO)$, whose orthonormal basis is given by $\{z^{i}w^{j}:i\ge0,j\ge-1\}$. A routine calculation shows that $[S_{z},S_{w}^*]=0$. Moreover, $\mathcal S$ is an example of a doubly commuting submodule of $\HT$ that properly contains $\HTO$.\eoex
\end{example}

In view of Theorem \ref{Doubly-com-sub}, we have the following.

\begin{corollary}\label{Rank}
The rank of a doubly commuting submodule of $\HT$ is $1.$
\end{corollary}

Let $T=(T_1, T_2)$ be a pair of commuting bounded linear operators on a Hilbert space $\mathcal{H}$. For a subset $E\subseteq\mathcal{H}$, we define $[E]_{T}$ as the closed subspace $$\overline{\text{span}}\left\{T_1^{k_{1}}T_2^{k_{2}}E:k_1,k_2\in\mathbb{Z}_+\right\}$$ of $\mathcal{H}$. The rank of $T$ is defined as the unique number
$$\text{rank}(T)=\text{min}\left\{\#E:[E]_{T}=\mathcal{H},E\subseteq\mathcal{H}\right\}.$$
We denote by $\#E$ the number of elements in $E$. Let $\mathcal{S}$ be a submodule of $\HT$. Then, the rank of $\mathcal{S}$ \cite{DP1989,III2011} is defined by
\begin{align*}
\mbox{rank}(\mathcal{S})=\mbox{rank}\left(S_{z},S_{w}\right).
\end{align*}
Note that Corollary \ref{Rank} concerns the cyclic submodules of $\HT$. In contrast, \cite[Proposition 5.1]{CJP2023} shows that the pair $(M_z, M_w)$ is not finitely cyclic on $\HT.$

\begin{proof}[Proof of Corollary \ref{Rank}]
Let $\mathcal S$ be a doubly commuting submodule of $\HT.$ Then by Theorem \ref{Doubly-com-sub}, $\mathcal S= q \HTO$ for some unimodular function $q$ on $\TRH.$ Since 
\beqn \mathcal{S}=\overline {\operatorname{span}}\{S_z^mS_w^nq : m,n  \ge 0\},\eeqn 
rank of $\mathcal S$ is $1.$
\end{proof}

Note that, $\HD\subset\HT$ (see \eqref{hdinht}), and every submodule of $\HD$ is also a submodule of $\HT$. This leads us to a natural question: Let $\mathcal{S}$ be a submodule of $\HT$. Is it necessarily true that $\Psi(\mathcal{S})$ (see \eqref{Psi-map} for the definition of $\Psi$) is a submodule of $\HD$?

In general, the answer is negative, as shown by the following example.

\begin{example}
Consider the submodule $$\mathcal{S}=\overline{\mathrm{span}}\Big\{\frac{1}{w}\left(\frac{z}{w}\right)^iw^j :0 \le  i \le j\Big\}\subset\HT.$$
It is easy to see that $\Psi(\mathcal{S})$ is not invariant under $M_1$.\eoex
\end{example}

The following result characterizes when the image of a submodule of $\HT$ under the map $\Psi$ is itself a submodule of $\HD$.

\begin{proposition} \label{image of a submodule}
Let $\mathcal S$ be a submodule of $\HT.$ Then $\Psi(\mathcal S)$ is a submodule of $\HD$ if and only if $\mathcal S$ is invariant under $M_{\varphi_1},$ where $\varphi_1$ is given as in \eqref{phi-map}.
\end{proposition}

\begin{proof}
It is easy to see that $\mathcal S$ is invariant under $M_w$ if and only if $\Psi(\mathcal S)$ is invariant under $M_2.$ Now, $\Psi(\mathcal S)$ is a submodule of $\HD$ if and only if $\Psi(\mathcal S)$ is invariant under $M_1$, which holds if and only if $\mathcal{S}$ is invariant under $M_{\varphi_1}$.
\end{proof}

\begin{proposition}\label{Psi-inverse-submodule}
For a submodule $\mathfrak{S}$ of $\HD$, $\Psi^{-1}(\mathfrak{S})$ is always a submodule of $\HT$.
\end{proposition}

\begin{proof}
Take $f\in\mathfrak{S}$. Then
$$z\Psi^{-1}(f)=z\cdot\frac{1}{w}\cdot f\circ \varphi=\Psi^{-1}(zwf)\in\Psi^{-1}(\mathfrak{S}),$$
and similarly,
$$w\Psi^{-1}(f)=\Psi^{-1}(wf)\in\Psi^{-1}(\mathfrak{S}).$$
This shows that $\Psi^{-1}(\mathfrak{S})$ is a submodule of $\HT$.
\end{proof}

\begin{corollary}
Let $\mathcal S$ be a doubly commuting submodule of $\HT$ that is invariant under $M_{\varphi_1}.$ Then $\Psi (\mathcal S)$ is a submodule of $\HD$ that is never doubly commuting.
\end{corollary}

\begin{proof}
Let $\mathcal S$ be a doubly commuting submodule of $\HT.$ By Theorem \ref{Doubly-com-sub}, $\mathcal{S}=q\HTO$ for some unimodular function $q$ on $\TRH$. Note that $$\Psi(\mathcal{S})=\left\{w(q\circ \varphi^{-1})\cdot f\circ\varphi^{-1} : f \in \HTO\right\}.$$
Since $\mathcal S$ is invariant under $M_{\varphi_1},$ in view of Proposition \ref{image of a submodule}, $\Psi (\mathcal S)$ is a submodule of $\HD.$ Now, assume, if possible, that $\Psi (\mathcal S)$ is a doubly commuting submodule of $\HD$. Then, by \cite[Theorem 2]{M1988}, we have
$$\Psi (\mathcal S)= \widetilde{q} \HD$$
for some inner function $\widetilde{q}$ on $\mathbb D^2.$ Therefore, we obtain $\mathcal S=\widetilde{q}\circ\varphi \HT,$ where $\widetilde{q}\circ\varphi$ is inner on $\TRH$. This leads to a contradiction (see Proposition \ref{NOT-Doubly-com-sub}).
\end{proof}

The rest of the section is devoted to exploring various properties of submodules of $\HT$. 

In what follows, we show that the multiplication operators $M_z$ and $M_w$ on $\HT$ cannot have a common reducing subspace.

\begin{proposition} \label{noreducingsubspace}
$\HT$ can not be decomposed as a direct sum of two proper submodules.
\end{proposition}

\begin{proof}
Suppose $\HT=\mathcal{S}\oplus\mathcal {S}^\perp$ for some submodule $\mathcal{S}$ such that  $\mathcal{S}^\perp$ is also a submodule. It follows that $\mathcal{S}$ is a reducing subspace of $\HT$. Therefore, by \cite[Corollary 6.3]{CJP2023}, the desired conclusion follows.
\end{proof}

Recall that two submodules $\mathcal{M}$ and $\mathcal{N}$ are said to have a positive angle (see \cite[p. 219]{DY2000}) if
\begin{align*}
\sup\left\{|\langle f,g\rangle|:f\in\mathcal{M},g\in\mathcal{N},\|f\|=\|g\|=1\right\}<1.    
\end{align*}
We claim that no two submodules can have a positive angle. This follows from the following lemma, together with an argument analogous to that used in \cite[Corollary 4.6]{DY2000}.

\begin{lemma}
Let $\mathcal S$ be a nontrivial submodule of $\HT.$ Then the joint minimal unitary dilation of $S_z$ and $S_w$ are the multiplications by $z$ and $w,$ respectively on $L^2(\mathbb T^2).$
\end{lemma}

\begin{proof}
Define $$\hat{\mathcal{S}} = \overline{\operatorname{span}}\{ z^i w^j f : f \in \mathcal{S}, \, i,j \in \mathbb{Z} \},$$ where the closure is taken in $L^{2}(\mathbb{T}^{2})$. Then the multiplication operators by $z$ and $w$ on $\hat{\mathcal{S}}$ are the joint minimal unitary dilation of the operators $S_z$ and $S_w$ acting on $\mathcal{S}$. Since $\hat{\mathcal{S}} \subseteq L^2(\mathbb{T}^2)$ and is invariant under multiplication by $z$, $w$, $\bar{z}$, and $\bar{w}$, it follows from \cite[Lemma 3]{GM1988} that $\hat{\mathcal{S}} = 1_B L^2(\mathbb{T}^2)$ for some measurable subset $B$ of $\mathbb{T}^2$. However, $\mathcal{S} \subseteq \hat{\mathcal{S}}$, and any nonzero function in $H^2(\partial_{d}(\TRH))$ can not vanish on a subset of $\mathbb{T}^2$ of positive measure. Therefore, we conclude that $\hat{\mathcal{S}} = L^2(\mathbb{T}^2)$.
\end{proof}

We conclude this section with a result that is of independent interest. Specifically, the following result shows that for any nonzero submodule $\mathcal{S}$ of $\HT,$ we have $\mathcal{S}\ominus z\mathcal{S}\not\perp\mathcal{S}\ominus w\mathcal{S}$ (cf. \cite{DeSa25} for an analogous result in the polydisc setting).

\begin{proposition}
Let $\mathcal{S}$ be a submodule of $\HT$. Then $\mathcal{S}=\{0\}$ if and only if 
$$(I_{\mathcal{S}}-S_{z}S_{z}^*)(I_{\mathcal{S}}-S_{w}S_{w}^*)=0.$$  
\end{proposition}

\begin{proof}
Let $(I_{\mathcal{S}}-S_{z}S_{z}^*)(I_{\mathcal{S}}-S_{w}S_{w}^*)=0$. Then, we have 
$$(I_{\mathcal{S}}-S_{w}S_{w}^*)\mathcal{S}\subset\text{ker}(I_{\mathcal{S}}-S_{z}S_{z}^*),$$
which is same as saying $\mathcal{S}\ominus w\mathcal{S}\subset z\mathcal{S}$. Consequently, we have $w^n(\mathcal{S}\ominus w\mathcal{S})\subset z\mathcal{S}$ for all $n\ge0$. From the Wold decomposition \cite{Wold,NaFo-Book}, we have $$\mathcal{S}=\bigoplus_{n=0}^{\infty}w^{n}(\mathcal{S}\ominus w\mathcal{S}),$$ and therefore we obtain $\mathcal{S}\subset z\mathcal{S}$, that is, $M_{z}^*\mathcal{S}\subset\mathcal{S}$. On the other hand, by taking adjoint in $(I_{\mathcal{S}}-S_{z}S_{z}^*)(I_{\mathcal{S}}-S_{w}S_{w}^*)=0$ we further obtain $M_{w}^*\mathcal{S}\subset\mathcal{S}$. This yields that $\mathcal S$ is a joint reducing subspace. Thus by Proposition \ref{noreducingsubspace}, we get $\mathcal{S}=\{0\}$ or $\mathcal{S}=\HT.$ $\mathcal{S}=\HT$ is not possible as 
\begin{align*}
\left(I_{\HT}-M_{z}M_{z}^*\right)\left(I_{\HT}-M_{w}M_{w}^*\right)\left(\frac{1}{w}\left(\frac{z}{w}\right)\right)\neq 0.
\end{align*}
This completes the proof.
\end{proof}
\section{Doubly commuting quotient modules of the form $(\theta\HT)^\perp$}\label{Sec3}
In this section, we initiate a study of the doubly commuting quotient modules of $\HT$ of the form $\mathcal{Q}_{\theta}=\HT \ominus \theta\HT$, with $\theta$  inner on $\TRH$ (see \cite[Definition 2.5]{JP2024}). We provide a complete characterization when $\theta\circ \varphi^{-1}$ is a product of two one-variable inner functions. We also show that several existing results concerning the quotient modules of $\HD$ do not hold in this context. The study of doubly commuting quotient modules of the Hardy module $\HD$ was initiated by Douglas and Yang in \cite{DY2000} and \cite{DY1998} (see also \cite{BCL1978}). Izuchi, Nakazi, and Seto later classified the doubly commuting quotient modules of $\HD$ in \cite[Theorem 2.1]{INS2004}. This result was subsequently generalized by Sarkar, who provided a complete characterization of the doubly commuting quotient modules of $H^2(\mathbb{D}^n)$ for $n \ge 2$ in \cite[Corollary 3.3]{Sarkar14}.

\cite[Theorem 2.1]{INS2004} provides a complete characterization of doubly commuting quotient modules in $\HD$. However, this result does not extend in the case of $\HT.$ For instance, the quotient module $\HT\ominus z^2\HT$ is not doubly commuting.

Let $\mathfrak{S}$ be a submodule of $\HD$. Then, by Proposition \ref{Psi-inverse-submodule}, $\Psi^{-1}(\mathfrak{S})$ is a submodule of $\HT$. Moreover,
\begin{align*}
f \in \HT \ominus \Psi^{-1}(\mathfrak{S})  &\iff \langle f,  \Psi^{-1}(h) \rangle_{\HT}=0, \quad h \in \mathfrak{S}\\
&\iff \langle \Psi(f), h \rangle_{\HD}=0,\quad h\in \mathfrak{S}. 
\end{align*}
Hence, 
\begin{align}\label{Transference}
\HT \ominus \Psi^{-1}(\mathfrak{S})=\Psi^{-1}(\HD\ominus\mathfrak{S}). 
\end{align}

The following result is an analog of \cite[Corollary 4.2]{DY2000} in the context of $\HT.$ To motivate the analog, we make a couple of remarks.
\begin{remark} 
Consider the quotient modules $\mathfrak{Q}=\HD \ominus\mathfrak{S}$ and $\mathcal{Q}=\HT \ominus\mathcal{S}$, where $\mathfrak{S}$ and $\mathcal{S}$ are submodules of $\HD$ and $\HT$, respectively. In \cite[Corollary 4.2]{DY2000}, Douglas and Yang proved that $\mathfrak{Q}$ is invariant under $M_1$ if and only if $\mathfrak{S}=\theta\HD$ for some inner function $\theta$ depending on $w$ only. In this situation, $\mathfrak{Q}$ is doubly commuting.  However, this fails for $\mathcal{Q}$, as explained below.
\begin{enumerate}
\item  Consider $\mathcal{S}=\theta\HT$, where $\theta(z,w)=w$. Then $\mathcal{Q}$ is doubly commuting, but fails to be invariant under $M_z.$
\item\label{Exam z inv quo} Consider the submodule $$\mathcal{S}=\overline{\operatorname{span}}\left\{\frac{1}{w}\left(\frac{z}{w}\right)^m w^n : m, n \in \mathbb Z_+, m <n\right\}\subset\HT.$$ Note that $\mathcal Q =\overline{\operatorname{span}}\left\{\frac{1}{w}\left(\frac{z}{w}\right)^m w^n : m \ge n\ge 0\right\}.$ Clearly, $\mathcal{Q}$ is invariant under $M_z.$ However, $\mathcal{Q}$ is not doubly commuting. Note also that $[S_z^*, S_w]=0$ on $\mathcal{S},$ which implies that although $\mathcal Q$ is $z$-invariant, the submodule  $\mathcal{S}$ is not of the form $\theta\HT$ for any inner function $\theta$ on $\triangle_H$ (see Proposition \ref{NOT-Doubly-com-sub}).  \eof
\end{enumerate}
\end{remark}

\begin{proposition}
Let $\mathcal{Q}_{\theta}$ be a quotient module for some inner function $\theta$ on $\TRH.$ Then $\mathcal{Q}_{\theta}$ is invariant under $M_{\varphi_1}$ if and only if $\theta$ is an inner function that depends only on $w$ $\rm ($see $\eqref{phi-map}$ for $\varphi_1$$\rm)$.
\end{proposition}

\begin{proof}
Let $\theta$ be an inner function that depends on $w$ only. Let $f\in\mathcal{Q}_{\theta}.$ First note that $$\left\langle\frac{z}{w}f,\theta \frac{1}{w}w^j\right\rangle=0\quad\mbox{for all $j\ge0$}.$$ Now for any $i\ge 1$ and $j\ge 0,$
$$\left\langle \frac{z}{w}f , \theta\frac{1}{w} \left(\frac{z}{w}\right)^iw^j\right\rangle=\left\langle f,\theta\frac{1}{w} \left(\frac{z}{w}\right)^{i-1}w^j\right \rangle=0.$$
Hence, $\frac{z}{w}f\in\mathcal{Q}_{\theta}$, showing that $\mathcal{Q}_{\theta}$ is invariant under $M_{\varphi_1}.$
	
Conversely, let $\mathcal{Q}_{\theta}$ be invariant under $M_{\varphi_1}.$ Arguing similarly to \eqref{Transference}, we obtain
\[\Psi(\mathcal Q_{\theta})=\HD \ominus\theta\circ \varphi^{-1}\HD,\] 
which is invariant under $M_1.$ An application of Remark \ref{hd-inner-ht} and \cite[Corollary 4.2]{DY2000} yields the conclusion.
\end{proof}

We refer the reader again to \eqref{Exam z inv quo}, which shows that even if $\HT\ominus\mathcal{S}$ is invariant under $M_{\varphi_1},$ the submodule $\mathcal{S}$ need not be of the form $\theta\HT$ for any inner function $\theta$ on $\triangle_{H}$. We now establish a result that plays a key role in our analysis.

\begin{lemma}\label{Relation_Quotient module actions}
Assume that $\mathfrak{S}$ is a submodule of $\HD$. Then  $\mathcal{Q}=\HT\ominus\Psi^{-1}(\mathfrak{S})$ is a quotient module of $\HT$ and
\beqn
Q_1Q_2=\Psi Q_z\Psi^{-1} \mbox{ ~~and~~ }Q_2=\Psi Q_w\Psi^{-1}.
\eeqn
\end{lemma}

\begin{proof}
Let $\mathfrak{Q}=\HD\ominus\mathfrak{S}.$ By Proposition \ref{Psi-inverse-submodule}, $\Psi^{-1}(\mathfrak{S})$ is a submodule of $\HT.$ It follows from \eqref{Transference} that $\mathcal{Q}=\Psi^{-1}(\mathfrak{Q})$. For any $f\in\HD$, we have
$$
\Psi P_{\mathcal{Q}} \Psi^{-1}(f)=\Psi P_{\mathcal{Q}} \Psi^{-1}((I-P_{\mathfrak{Q}})f+P_{\mathfrak{Q}}f)=\Psi P_{\mathcal{Q}} \Psi^{-1}(P_{\mathfrak{Q}}f)=\Psi\Psi^{-1}(P_{\mathfrak{Q}}f)=P_{\mathfrak{Q}}f,
$$ 
where $P_\mathcal{Q}:\HT\to\mathcal{Q}$ and $P_{\mathfrak Q}:\HD\to\mathfrak{Q}$ are the orthogonal projections on $\mathcal{Q}$ and $\mathfrak{Q}$, respectively. Thus we have
\beq \label{Projection relations}
\Psi P_\mathcal{Q}\Psi^{-1}=P_{\mathfrak{Q}}.
\eeq
Now for any $f\in\mathfrak{Q}$, we have
\beqn
\Psi Q_{w}\Psi^{-1}(f)=\Psi P_{\mathcal{Q}}M_{w}\Psi^{-1}(f)\overset{\eqref{Diagram}}{=}\Psi P_{\mathcal{Q}}\Psi^{-1} M_{2}\Psi\Psi^{-1}(f)\overset{\eqref{Projection relations}}{=}P_{\mathfrak{Q}}M_{2}f=Q_{2}f
\eeqn
and 
\begin{align*}
\Psi Q_{z}\Psi^{-1}(f)=\Psi P_{\mathcal{Q}}M_{z}\Psi^{-1}(f)
&\overset{\eqref{Diagram}}{=}\Psi P_{\mathcal{Q}}\Psi^{-1}M_1M_2\Psi\Psi^{-1}(f)\\
&\overset{\eqref{Projection relations}}{=}P_{\mathfrak{Q}}M_1M_2f\\
&=P_{\mathfrak{Q}}M_1M_2f-P_{\mathfrak{Q}}M_1P_{\mathfrak{S}}M_2f\\
&=P_{\mathcal{Q}}M_{1}P_{\mathcal{Q}}M_{2}f=Q_{1}Q_{2}f.
\end{align*}
This completes the proof.
\end{proof}

Under the assumptions of Lemma \ref{Relation_Quotient module actions}, we have the following identity: 
\begin{align*}
[Q_{1}Q_{2},Q_{2}^*]=\Psi[Q_{z},Q_{w}^*]\Psi^{-1}.
\end{align*}
As a consequence, we obtain the equivalence:
\begin{align}\label{Equivalent_DC}
\mbox{$[Q_{1}Q_{2},Q_{2}^*]=0$ on $\mathfrak{Q}$ if and only if $[Q_{z},Q_{w}^*]=0$ on $\Psi^{-1}(\mathfrak{Q})$}.
\end{align} 
In view of \eqref{Equivalent_DC}, we investigate those quotient modules  $\mathfrak{Q}$ of $\HD$ that satisfy
$$[Q_1 Q_2, Q_2^*] = 0,$$
which further helps in characterization of certain doubly commuting quotient modules of $ \HT$ (see Theorem \ref{Hartogscharasepa}). To this end, we introduce the notion of $\varphi$-doubly commuting quotient modules of $\HD$.

\begin{definition}\label{phi-dc}
A quotient module $\mathfrak{Q}$ of $\HD$ is said to be  $\varphi$-\textit{doubly commuting} if $(Q_1Q_2, Q_2)$ is doubly commuting, that is,
$$Q_{1}Q_{2}Q_{2}^*=Q_{2}^*Q_{1}Q_{2}.$$
\end{definition}

\subsection{Characterization of $\varphi$-doubly commuting quotient modules of $\HD$}
The Hardy space on the bidisc admits the tensor product decomposition
$$H^2(\mathbb{D}^2)\cong H^2(\mathbb{D})\otimes H^2(\mathbb{D}),$$
with $z^m w^n$ is identified as $z^m \otimes w^n$, and $H^2(\mathbb{D})$ denotes the Hardy space of the unit disc. Let $T_z$ and $T_w$ denote the multiplication operators on $H^2(\mathbb{D})$ corresponding to the variables $z$ and $w$, respectively. Then the coordinate multiplication operators $M_1$ and $M_2$ on $H^2(\mathbb{D}^2)$ act as
$$M_1 = T_z \otimes I\quad\mbox{and}\quad M_2 = I \otimes T_w.$$ 
Next, consider a quotient module $\mathfrak{Q}=\mathfrak{Q}_{1}\otimes\mathfrak{Q}_{2}$ of $\HD$, where $\mathfrak{Q}_{1}$ and $\mathfrak{Q}_{2}$ are quotient modules of $\HDO$. Following \cite[p.~3]{Sarkar14}, we have the module multiplication operators on $\mathfrak{Q}$ as
\begin{align}\label{Q_1Q_2}
Q_1=P_{\mathfrak{Q}_{1}}T_{z}\big|_{\mathfrak{Q}_{1}}\otimes I_{\mathfrak{Q}_{2}}\quad\mbox{and}\quad Q_{2}=I_{\mathfrak{Q}_{1}}\otimes P_{\mathfrak{Q}_{2}}T_{w}\big|_{\mathfrak{Q}_{2}}.
\end{align}
Then
\begin{align}\label{[Q_2,Q_2*]}
\nonumber[Q_{2},Q_{2}^*]&=I_{\mathfrak{Q}_{1}}\otimes\left[P_{\mathfrak{Q}_{2}}T_{w}T_{w}^*\big|_{\mathfrak{Q}_{2}}-P_{\mathfrak{Q}_{2}}T_{w}^*P_{\mathfrak{Q}_{2}}T_{w}\big|_{\mathfrak{Q}_{2}}\right]\\
&=I_{\mathfrak{Q}_{1}}\otimes\left[P_{\mathfrak{Q}_{2}}T_{w}\big|_{\mathfrak{Q}_{2}},T_{w}^*\big|_{\mathfrak{Q}_{2}}\right],
\end{align}
which along with \eqref{Q_1Q_2} gives
\begin{align}\label{Q_1[Q_2,Q_2*]}
Q_{1}[Q_{2},Q_{2}^*]=P_{\mathfrak{Q}_{1}}T_{z}\big|_{\mathfrak{Q}_{1}}\otimes\left[P_{\mathfrak{Q}_{2}}T_{w}\big|_{\mathfrak{Q}_{2}},T_{w}^*\big|_{\mathfrak{Q}_{2}}\right].
\end{align}
From now on, we will use \eqref{Q_1Q_2}, \eqref{[Q_2,Q_2*]} and \eqref{Q_1[Q_2,Q_2*]}, along with the usual definition of $\HD$, interchangeably wherever necessary.

\begin{lemma}\label{Q2 normal}
Let $\mathfrak{Q}=\mathfrak{Q}_{1}\otimes\mathfrak{Q}_{2}$ be a nontrivial quotient module of $\HD$, where $\mathfrak{Q}_{1}$ and $\mathfrak{Q}_{2}$ are quotient modules of $H^{2}(\mathbb{D}).$ Then $Q_{2}$ is normal if and only if $\mathfrak{Q}_{2}$ is one dimensional.
\end{lemma}

\begin{proof}
Let $Q_{2}$ be normal. From \eqref{[Q_2,Q_2*]}, it follows that
$$\left[P_{\mathfrak{Q}_{2}}T_{w}\big|_{\mathfrak{Q}_{2}},T_{w}^*\big|_{\mathfrak{Q}_{2}}\right]=0,$$
which due to \cite[Lemma 2.4]{DaGoSa20} further implies that $\mathfrak{Q}_{2}$ is one dimensional.

Conversely, let $\mathfrak{Q}_{2}$ be one dimensional. Then, by \cite[Lemma 2.4]{DaGoSa20} and \eqref{[Q_2,Q_2*]}, it follows that $Q_{2}$ is normal. This completes the proof.
\end{proof}

For $a\in\mathbb{D},$ let $f_{a}$ denote the automorphism of the unit disc $\mathbb D$ given by 
\beq \label{automorphismdisc}
f_a(z)=c\frac{z-a}{1-\bar a z}, \quad z \in \mathbb{D}, c\in\mathbb{T}.
\eeq
For future reference, define the functions $k_a(z, w)$ and $m_a(z, w)$ by
\begin{align}\label{R-Kernel}
k_{a}(z,w)=\frac{1}{1-\bar{a}z} \mbox{ and } m_{a}(z,w)=\frac{1}{1-\bar{a}w} ,\quad (z,w)\in\mathbb{D}^{2},
\end{align}
where $a \in \mathbb D.$

If $\mathfrak{Q}$ is a doubly commuting quotient module of $\HD$, then
\beq\label{Doubly-Commute}
[Q_{1}Q_{2},Q_{2}^*] =Q_1Q_2Q_2^*-Q_2^*Q_1Q_2=Q_1[Q_2, Q_2^*].
\eeq
Note that if $[Q_2, Q_2^*]=0$ or $\operatorname{ran}[Q_2, Q_2^*] \subseteq \operatorname{ker}Q_1$, then by \eqref{Doubly-Commute}, it follows that $\mathfrak{Q}$ is $\varphi$-doubly commuting. The following examples illustrate that both conditions can occur.

\begin{example}
Consider the submodule $\mathfrak{S}_1=\overline{\operatorname{span}}\left\{{z}^{i}w^{j}:i, j\ge 0,i+j\ge 1\right\}$. Then, $\mathfrak{Q}_1=\HD\ominus\mathfrak{S}_1$ is spanned by $1$ and is doubly commuting. Moreover, $[Q_2, Q_2^*]=0$. Further, consider the submodule $\mathfrak{S}_2=z\HD.$ Then, $\mathfrak{Q}_{2}=\overline{\operatorname{span}}\left\{w^n: n \ge 0\right\}.$ Note that $\mathfrak{Q}_2$ is doubly commuting and $[Q_2, Q_2^*]1\neq0,$ but $\text{ker}(Q_{1})=\mathfrak{Q}_{2}$.\eoex 
\end{example}

In what follows, we characterize the $\varphi$-doubly commuting quotient modules of $\HD$ among all doubly commuting quotient modules of $\HD$. We begin with an example of a doubly commuting quotient module which is not $\varphi$-doubly commuting.

\begin{example}\label{Exam not phi DC}
Consider the quotient module $\mathfrak{Q}=\HD\ominus f_{a}(z)\HD$ for some $a\in\mathbb{D}$ and $a\neq0$. By \cite[Theorem 2.1]{INS2004}, it follows that $\mathfrak{Q}$ is doubly commuting. Note that $k_{a} \in\mathfrak{Q}$ (see \eqref{R-Kernel}). By \cite[Lemma 12]{CIL2018}, it follows that $wk_{a }\in\mathfrak{Q}$. Now,
$$[Q_{1}Q_{2},Q_{2}^*]k_{a}\overset{\eqref{Doubly-Commute}}{=}(Q_1Q_2Q_2^*-Q_1Q_2^*Q_2)k_a=-Q_1Q_2^*Q_2k_a=-Q_1k_{a},$$ 
which is nonzero since $a \neq 0.$ Hence, $\mathfrak{Q}$ is not $\varphi$-doubly commuting.\eoex 
\end{example}

\begin{theorem} \label{DC+varDC}
Let $\mathfrak{Q}=\mathfrak{Q}_{1}\otimes\mathfrak{Q}_{2}$ be a nontrivial doubly commuting quotient module of $\HD.$ Then $\mathfrak{Q}$ is $\varphi$-doubly commuting if and only if one of the following holds:
\begin{itemize}
\item [$\mathrm{(i)}$] $\mathfrak{Q}_1=\HDO$ and $\mathfrak{Q}_2=\HDO \ominus f_a(w)\HDO$
\item [$\mathrm{(ii)}$] $\mathfrak{Q}_1=\HDO \ominus f_0(z)\HDO$ and $\mathfrak{Q}_2=\HDO$ 
\item [$\mathrm{(iii)}$] $\mathfrak{Q}_1=(\HDO \ominus \theta_{1}(z)\HDO)$ and $\mathfrak{Q}_2=(\HDO \ominus \theta_{2}(w)\HDO),$ where either $\theta_{1}(z)=f_0(z)$ and $\theta_2$ is an arbitrary inner function, or $\theta_1$ is arbitrary and $\theta_{2}(w)=f_a(w)$ for some $a\in\mathbb{D}$.
\end{itemize}
\end{theorem}

\begin{proof}
We prove only the necessary part, as the sufficient part is immediate. Let $\mathfrak{Q}$ be a nontrivial doubly commuting quotient module of $\HD$. From \cite[Theorem 2.1]{INS2004} (see also, \cite[Corollary 3.3]{Sarkar14}), one of the following four cases occurs: 
\begin{itemize}
\item [$\mathrm{(i)}$]$\mathfrak{Q}=\HDO \otimes\HDO$
\item [$\mathrm{(ii)}$]$\mathfrak{Q}=\HDO \otimes (\HDO \ominus \theta_2(w)\HDO)$
\item [$\mathrm{(iii)}$]$\mathfrak{Q}=(\HDO \ominus \theta_1(z)\HDO) \otimes \HDO$
\item [$\mathrm{(iv)}$]$\mathfrak{Q}=(\HDO \ominus \theta_1(z)\HDO) \otimes (\HDO \ominus \theta_2(w)\HDO),$
\end{itemize}
where $\theta_1$ and $\theta_2$ are nonconstant inner functions on $\mathbb D.$  Note that for (i), we have
\[[Q_{1}Q_{2},Q_{2}^*](1\otimes 1)\neq0.\]
Hence $\mathfrak{Q}=\HDO \otimes\HDO$ is not $\varphi$-doubly commuting.

We now consider the remaining cases individually. Let $\mathfrak{Q}$ be $\varphi$-doubly commuting (see Definition \ref{phi-dc}).

\noindent \textsf{Case $\mathrm{(ii)}$:$ ~\mathfrak{Q}=\HDO \otimes (\HDO \ominus \theta_2(w)\HDO).$}

Then $\mathfrak{Q}$ is invariant under $M_1.$ This yields that $\operatorname{ker}(Q_1)=\{0\}.$ Since $\mathfrak{Q}$ is $\varphi$-doubly commuting, by \eqref{Doubly-Commute} and \eqref{Q_1[Q_2,Q_2*]}
\beqn Q_1[Q_2 ,Q_2^*]( 1 \otimes f)=0\mbox{ yields that } \left[P_{\mathfrak{Q}_{2}}T_{w}\big|_{\mathfrak{Q}_{2}},T_{w}^*\big|_{\mathfrak{Q}_{2}}\right]f=0 \mbox{ for all }f\in\mathfrak{Q}_2.\eeqn
Consequently, $P_{\mathfrak{Q}_{2}}T_{w}\big|_{\mathfrak{Q}_{2}}$ is normal on $\mathfrak{Q}_2$. This with \eqref{[Q_2,Q_2*]} yields that $Q_2$ is normal. Therefore, by Lemma \ref{Q2 normal}, $\HDO \ominus \theta_2(w)\HDO$ is one dimensional. Furthermore, by \cite[Propositions 5.19 and 5.16]{GMR2016}, it follows that $\theta_{2}=f_a$ for some $a\in\mathbb{D}$ (see \eqref{automorphismdisc}).

\noindent \textsf{Case $\mathrm{(iii)}$:~$\mathfrak{Q}=(\HDO \ominus \theta_1(z)\HDO) \otimes \HDO.$}

Then $\mathfrak{Q}$ is invariant under $M_2.$ Thus, using \eqref{[Q_2,Q_2*]}, we get
\begin{align*}[Q_{2},Q_{2}^*] =I_{\mathfrak{Q}_{1}}\otimes\left[T_{w},T_{w}^*\right].
\end{align*}
It follows that for any $f\otimes 1\in \mathfrak{Q},$ we have
\begin{align}\label{DC+varDC-R22}
Q_1[Q_2 ,Q_2^*](f\otimes 1)=Q_1(I_{\mathfrak{Q}_{1}}\otimes\left[{T_{w},T^*_{w}}\right])(f\otimes 1)=Q_{1}(f\otimes-1).
\end{align}
Since $\mathfrak{Q}$ is $\varphi$-doubly commuting, it follows from \eqref{DC+varDC-R22} that $P_{\mathfrak{Q}_{1}}T_{z}\big|_{\mathfrak{Q}_{1}} $ is a zero operator on $\HDO\ominus\theta_1(z)\HDO.$ This with \cite[Proposition 9.13]{GMR2016} yields that $\HDO\ominus\theta_{1}(z)\HDO$ is at most one dimensional. Since $\theta_1$ is nonconstant, it follows from \cite[Proposition 5.19]{GMR2016} that $\theta_{1}=f_a$ for some $a\in\mathbb{D}$. Note that, if $a\neq0$, then Example \ref{Exam not phi DC} shows that $\mathfrak{Q}$ is not $\varphi$-doubly commuting. Hence $\theta_{1}=f_0.$

\noindent \textsf{Case $\mathrm{(iv)}$:} $\mathfrak{Q}=(\HDO \ominus \theta_1(z)\HDO) \otimes (\HDO \ominus \theta_2(w)\HDO).$\\ 
Now, the $\varphi$-doubly commutativity of $\mathfrak{Q}$ and \eqref{Q_1[Q_2,Q_2*]} implies that for each $f\otimes g\in\mathfrak{Q}_{1}\otimes\mathfrak{Q}_{2},$ 
$$P_{\mathfrak{Q}_{1}}T_{z}f=0\quad\mbox{or}\quad\left[P_{\mathfrak{Q}_{2}}T_{w}\big|_{\mathfrak{Q}_{2}},T_{w}^*\big|_{\mathfrak{Q}_{2}}\right]g=0.$$
If $\left[P_{\mathfrak{Q}_{2}}T_{w}\big|_{\mathfrak{Q}_{2}},T_{w}^*\big|_{\mathfrak{Q}_{2}}\right]=0\quad\text{on}~\mathfrak{Q}_{2},$ then it follows from Lemma \ref{Q2 normal} that $\theta_{2}=f_a$ for some $a\in\mathbb{D}$. In this case, $\theta_{1}$ can be any inner function. We now consider the situation when there exists $\tilde g \in \mathfrak{Q}_2$ such that $\left[P_{\mathfrak{Q}_{2}}T_{w}\big|_{\mathfrak{Q}_{2}},T_{w}^*\big|_{\mathfrak{Q}_{2}}\right]\tilde g\neq0.$ Then $$Q_{1}[Q_{2},Q_{2}^*](f\otimes \tilde g)=P_{\mathfrak{Q}_{1}}T_{z}f\otimes\left[P_{\mathfrak{Q}_{2}}T_{w}\big|_{\mathfrak{Q}_{2}},T_{w}^*\big|_{\mathfrak{Q}_{2}}\right]\tilde g=0$$
implies that $P_{\mathfrak{Q}_{1}}T_{z}\big|_{\mathfrak{Q}_{1}}$ is a zero operator on $\mathfrak{Q}_{1}=\HDO \ominus \theta_1(z)\HDO$. Now, by following the arguments used in \textsf{Case $\mathrm{(iii)}$}, we obtain $\theta_{1}=f_{0}.$ In this case, the inner function $\theta_{2}$ is arbitrary. This completes the proof.
\end{proof}

Recall that an inner function on $\mathbb D$ is called singular if it has no zeros in $\mathbb{D}$. The following is an application of Theorem \ref{DC+varDC}.

\begin{corollary}\label{keycoro1}
Let $S_1$ and $S_2$ be singular inner functions on $\mathbb{D} $.  Then $\mathfrak{Q}=\HD \ominus S_1(z)S_2(w)\HD$ is $\varphi$-doubly commuting if and only if $S_1$ and $S_2$ both are constant inner functions.
\end{corollary}

\begin{proof}
Let $\mathfrak{Q}$ be $\varphi$-doubly commuting. If possible, let $S_1$ be a nonconstant inner function. Consider the function 
$f(z,w)=Q_1^*(S_1(z)-S_{1}(0)).$
Since $S_1$ is nonconstant, $f\neq 0.$ Moreover, $f \in \mathfrak{Q}.$ Note that $Q_1Q_2Q_2^*f(z,w)=0.$ On the other hand,
\beqn
Q_2^*Q_2Q_1f(z,w)=Q_2^*Q_2P_{\mathfrak{Q}}\left[S_1(z)-S_1(0)\right]
=S_1(z)-S_1(0) \neq 0,
\eeqn
since $S_1$ is nonconstant. We arrive at a contradiction (see Definition \ref{phi-dc}).

Now, assume if possible that $S_1$ is constant and $S_2$ is nonconstant. Then $\mathfrak{Q} = \HD \ominus S_2(w)\HD$ is doubly commuting, which is not $\varphi$-doubly commuting by Theorem \ref{DC+varDC}. Hence $S_1$ and $S_2$ are constant inner functions.
\end{proof}

The following example illustrates that not all $\varphi$-doubly commuting quotient modules of $\HD$ are doubly commuting.

\begin{example}
Let $\mathfrak{S}=\overline{\text{span}}\left\{z^{i}w^{j}:i, j\ge 0,i+j\ge 2\right\}$.  Then $\mathfrak{S}$ is a submodule of $\HD$. Moreover, the corresponding quotient module $\mathfrak{Q}$ is $\overline{\text{span}}\{1,z,w\}$. It is easy to see that $\mathfrak{Q}$ is $\varphi$-doubly commuting. However, $\mathfrak{Q}$ is not doubly commuting.\eoex
\end{example}

The following result characterizes $\varphi$-doubly commuting quotient modules of the form $\HD\ominus\theta_1(z)\theta_2(w)\HD,$ where $\theta_1$ and $\theta_2$ are inner functions on $\mathbb D.$ These quotient modules are not doubly commuting (see \cite[Theorem 2.1]{INS2004}) when $\theta_1$ and $\theta_2$ are nonconstant inner functions.

\begin{theorem}\label{finalcharasepa}
For inner functions $\theta_{1}$ and $\theta_{2}$ on $\mathbb{D},$ the quotient module $\mathfrak{Q}=\HD\ominus \theta_1(z)\theta_2(w)\HD$ is $\varphi$-doubly commuting if and only if one of the following holds:
\begin{itemize}
\item [$\rm{(i)}$] $\theta_{1}$ and $\theta_{2}$ are constant functions.
\item [$\rm{(ii)}$] $\theta_1$ is constant and $\theta_2=f_b$ for some $b \in \mathbb D.$
\item [$\rm{(iii)}$] $\theta_2$ is constant and $\theta_1(z)=c_1z$ for some $c_1 \in \mathbb{T}.$ 
\item [$\rm{(iv)}$] $\theta_1(z)=c_1z$ and $\theta_2(w)=c_2w$ for some $c_1, c_2\in \mathbb{T}.$
\end{itemize}
\end{theorem}

We begin by proving a couple of technical lemmas that will play a key role in the characterization.

\begin{lemma}\label{Keylemma1}
Consider the quotient module $\mathfrak{Q}=\HD \ominus \theta_1(z)\theta_2(w)\HD$, where $\theta_1$ and $\theta_2$ are inner functions on $\mathbb{D}$ such that $\theta_1(a)=0$ for some $a\in\mathbb{D},$ with $a\neq0$. Then, $\mathfrak{Q}$ is not $\varphi$-doubly commuting.
\end{lemma}

\begin{proof}
Since $\theta_1(a)=0,$ there exists an inner function $\theta_3$ on $\mathbb  D$ such that $\theta_1=f_a\theta_3$ (see \eqref{automorphismdisc}). We now present the proof by considering separate cases on $\theta_2.$

\noindent\textsf{Case (i): Assume that $\theta_2$ is a constant function.} 

Consider the function $f(z, w)=k_a(z, w)\theta_3(z)$ (see \eqref{R-Kernel}). It is easy to see that $f \in \mathfrak{Q}.$ Note that $Q_1Q_2Q_2^*f(z,w)=0.$ Whereas
\beqn
Q_2^*Q_1Q_2f(z,w)=Q_2^*Q_1Q_2\left( k_a(z, w)\theta_3(z))\right) 
&=&Q_2^*P_{\mathfrak{Q}}\left( zwk_a(z, w)\theta_3(z)\right)
\\&=&ak_a(z, w)\theta_3(z)) \neq 0,
\eeqn
since $a \neq 0.$ Hence $\mathfrak{Q}$ is not $\varphi$-doubly commuting.

\noindent\textsf{Case (ii): $\theta_2$ is nonconstant and $\theta_2(0)\neq 0.$}

Consider the function $f(z, w)=k_a(z, w)\theta_3(z)\theta_2(w)$. It is evident that $f \in \mathfrak{Q}.$ Now 
\beqn
Q_2^*Q_1Q_2f(z,w)&=&Q_2^*Q_1Q_2\left( k_a(z, w)\theta_3(z)\theta_2(w)\right)
\\&=&Q_2^*P_{\mathfrak{Q}}\left( zwk_a(z, w)\theta_3(z)\theta_2(w)\right)
\\&=&ak_a(z, w)\theta_3(z)\theta_2(w).
\eeqn  
On the other hand,
\beqn
Q_1Q_2Q_2^*f(z,w)&=&Q_1Q_2Q_2^*\left( k_a(z, w)\theta_3(z)\theta_2(w)\right) 
\\&=&Q_1\left( k_a(z, w)\theta_3(z)\theta_2(w)-k_a(z, w)\theta_3(z)\theta_2(0)\right) 
\\&=&P_{\mathfrak{Q}}\left( zk_a(z, w)\theta_3(z)\theta_2(w)-zk_a(z, w)\theta_3(z)\theta_2(0)\right) 
\\&=& ak_a(z, w)\theta_3(z)\theta_2(w)-P_{\mathfrak{Q}}\left(zk_a(z, w)\theta_3(z)\theta_2(0)\right). 
\eeqn 
Thus 
\beqn
Q_2^*Q_1Q_2f(z,w)-Q_1Q_2Q_2^*f(z,w)=\theta_2(0)P_{\mathfrak{Q}}\left(zk_a(z, w)\theta_3(z)\right). 
\eeqn
Since $\theta_2(0)\neq 0$ and $zk_a(z, w)\theta_3(z) \notin \theta_1(z)\theta_2(w)\HD,$ the quotient module $\mathfrak{Q}$ is not $\varphi$-doubly commuting.

\noindent\textsf{Case (iii): $\theta_2$ is nonconstant and there exist $n \in \mathbb N$ and an inner function $\theta_4$ on $\mathbb D$ such that $\theta_2(w)=w^n\theta_4(w)$ and $\theta_4(0)\neq 0.$}  

It is easy to see that the functions $k_{a}(z,w)\theta_{3}(z)\theta_4(w), wk_{a}(z,w)\theta_{3}(z)\theta_4(w)$ are in $\mathfrak{Q}.$ Now,
\beq \label{n=1nge2cases}
Q_{2}^*Q_{1}Q_{2}k_{a}(z,w)\theta_{3}(z)\theta_4(w)=\begin{cases} 
ak_{a}(z,w)\theta_{3}(z)\theta_4(w) & \text{if } n=1, \\
zk_{a}(z,w)\theta_{3}(z)\theta_4(w) & \text{if } n\ge 2. 
\end{cases}
\eeq
Now 
\begin{align*}
Q_{1}Q_{2}Q_{2}^*k_{a}(z,w)\theta_{3}(z)\theta_4(w)&=Q_{1}P_{\mathfrak{Q}}\left(k_{a}(z,w)\theta_{3}(z)\theta_4(w)-k_{a}(z,w)\theta_{3}(z)\theta_4(0)\right)\\
&=zk_{a}(z,w)\theta_{3}(z)\theta_4(w)-zk_{a}(z,w)\theta_{3}(z)\theta_4(0)
\end{align*}
This, together with \eqref{n=1nge2cases} and the fact that $\theta_4(0)\neq 0$, yields that $\mathfrak{Q}$ is not $\varphi$-doubly commuting. This completes the proof.
\end{proof}

\begin{lemma} \label{keylemma2}
Let $\theta_1$ and $\theta_2$ be inner functions on $\mathbb D$ such that $\theta_1(z) = z^n S_1(z)$, where $S_1(z)$ is a singular inner function and $n \in \mathbb N.$ In addition, for $n = 1$, we assume that $S_1(z)$ is nonconstant. Then, $\mathfrak{Q}=\HD \ominus \theta_1(z)\theta_2(w)\HD$ is not $\varphi$-doubly commuting.
\end{lemma}

\begin{proof} 
The proof of the lemma is carried out by dividing into cases.

\noindent\textsf{Case (i): Let $\theta_2$ be a constant inner function and $n\ge2.$}  

Consider the function $f(z,w)=S_1(z).$ It is easy to see that $f \in \mathfrak{Q}$ and $Q_1Q_2 Q_2^* f=0,$ whereas
\beqn
Q_2^*Q_1Q_2f(z,w)=Q_2^*Q_1\left( wS_1(z)\right) 
=Q_2^*P_{\mathfrak{Q}}\left( zwS_1(z)\right) =zS_1(z).
\eeqn  
Hence $\mathfrak{Q}$ is not $\varphi$-doubly commuting.

\noindent\textsf{Case (ii): Let $n = 1$, and assume that either $\theta_2$ is a constant inner function or $\theta_2(w) = w$. }  

Consider the function $f(z, w)=1.$ Then $f \in \mathfrak{Q}$ and $Q_1Q_2Q_2^*f(z,w)=0$, whereas
\beqn
Q_2^*Q_1Q_2f(z,w)=Q_2^*Q_1\left( w\right) 
=Q_2^*P_{\mathfrak{Q}}\left(zw\right) \neq 0,
\eeqn  
since $zw \notin \mathfrak{Q}$ and $zw \notin \theta_1(z)\theta_2(w) \HD.$ Hence $\mathfrak{Q}$ is not $\varphi$-doubly commuting.

\noindent\textsf{Case (iii): $\theta_2$ is nonconstant and $\theta_2(0)\neq 0.$}  

Consider $f(z, w)=z^{n-1}S_1(z)\theta_2(w) \in \mathfrak{Q}.$ It is easy to see that $Q_2^*Q_1Q_2f=0.$ On the other hand,
 \beqn
Q_1Q_2Q_2^*f(z,w)&=&Q_1Q_2Q_2^*\left(z^{n-1}S_1(z)\theta_2(w) \right)\\
&=&Q_1P_{\mathfrak{Q}}\left(z^{n-1}S_1(z)\theta_2(w)-z^{n-1}S_1(z)\theta_2(0) \right)
\\
&=&-P_{\mathfrak{Q}} \left(z^{n}S_1(z)\theta_2(0)\right),
\eeqn
which is nonzero as $\theta_2(0)\neq 0$ and $z^{n}S_1(z) \notin\theta_1(z)\theta_2(w)\HD.$
It follows that $\mathfrak{Q}$ is not $\varphi$-doubly commuting.

\noindent\textsf{Case (iv): Either $\theta_2(w)=w^m$, where $m\ge2$, or $\theta_{2}(w)=w$ with $n\ge2$ }

In this case, we have $1,z,w,zw\in\mathfrak{Q}$. Therefore, $\mathfrak{Q}$ is not $\varphi$-doubly commuting.

\noindent\textsf{Case (v): Assume that $\theta_2$ is nonconstant and there exists a nonconstant inner function $\theta_{3}$ such that $\theta_2(w)=w^m\theta_{3}(w),$ where $\theta_{3}(0)\neq 0$ and $m\in\mathbb{N}.$ } 

If $m>1$ or $n>1$, then $\mathfrak{Q}$ fails to be $\varphi$-doubly commuting as $zw \in \mathfrak{Q}.$ So let $m=n=1.$ Then, it is evident that $S_{1}(z)\theta_{3}(w),wS_{1}(z)\theta_{3}(w)\in\mathfrak{Q}$. Moreover, observe that $Q_{2}^*Q_{1}Q_{2}S_{1}(z)\theta_{3}(w)=0,$ while
\begin{align*}
Q_{1}Q_{2}Q_{2}^*S_{1}(z)\theta_{3}(w)&=Q_{1}P_{\mathfrak{Q}}\left(S_{1}(z)\theta_{3}(w)-S_{1}(z)\theta_{3}(0)\right)\\
&=P_{\mathfrak{Q}}\left(zS_{1}(z)\theta_{3}(w)-zS_{1}(z)\theta_{3}(0)\right)\\
&=zS_{1}(z)\theta_{3}(w)-zS_{1}(z)\theta_{3}(0).
\end{align*}
Since $\theta_{3}$ is nonconstant, it follows that $\mathfrak{Q}$ is not $\varphi$-doubly commuting.
\end{proof}

\begin{lemma}\label{keylemma3}
Let $S_1$ be a nonconstant singular inner function and $\theta_2$ be an inner function such that $ \theta_2(b)=0$ for some $b\in \mathbb{D}$. Then $\mathfrak{Q}=\HD \ominus S_1(z)\theta_2(w)\HD$ is not $\varphi$-doubly commuting.
\end{lemma}

\begin{proof}
Since $ \theta_2(b)=0,$ there exists an inner function $\theta_3$ such that $\theta_2(w)=f_b(w)\theta_3(w)$ (see \eqref{automorphismdisc}). We now proceed with the proof by considering separate cases.

\noindent\textsf{Case (i): Let $\theta_3$ be a constant inner function and $b=0.$} 

Note that $Q_1Q_2Q_2^*(1)=0.$ Whereas, 
\beqn
Q_2^*Q_2Q_1(1)=Q_2^*Q_2(z)=Q_2^*P_{\mathfrak{Q}}(zw) \neq 0,
\eeqn
since $S_1$ is nonconstant. Thus $\mathfrak{Q}$ is not $\varphi$-doubly commuting.

\noindent\textsf{Case (ii): $b\neq0$ and $\theta_3$ has a zero other than $b.$ }

Consider the function $f(z, w)=m_b(z, w)S_1(z)\theta_3(w)$ (see \eqref{R-Kernel}). Clearly, $f \in \mathfrak{Q}.$
Now 
\beq  \label{theta1singu1}
Q_2^*Q_2Q_1f(z, w)\notag&=&Q_2^*Q_2Q_1 \left(m_b(z, w)S_1(z)\theta_3(w)\right)  \\
\notag&=&Q_2^*P_{\mathfrak{Q}} \left(zwm_b(z, w)S_1(z)\theta_3(w)\right) \\
&=& \frac{bzS_1(z)}{w}\left[m_b(z, w)\theta_3(w)-\theta_3(0)\right].
\eeq
On the other hand,
\beq \label{theta1singu2}
Q_1Q_2Q_2^*f(z, w)\notag&=&Q_1Q_2Q_2^* \left(m_b(z, w)S_1(z)\theta_3(w)\right) \\
\notag&=&  Q_1\left(S_1(z)m_b(z, w)\theta_3(w)-S_1(z)\theta_3(0)\right)\\
&=& zS_1(z)m_b(z, w)\theta_3(w)- P_{\mathfrak{Q}}\left(zS_1(z)\theta_3(0)\right).
\eeq

\noindent\textsf{Subcase (a): $\theta_3(0) = 0$} 

Then from \eqref{theta1singu1} and \eqref{theta1singu2}, we have
$$Q_2^*Q_2Q_1f(z, w)= bzS_1(z)\dfrac{m_b(z, w)\theta_3(w)}{w}$$ 
and $ Q_1Q_2Q_2^*f(z, w)=zS_1(z)m_b(z, w)\theta_3(w).$ Thus $\mathfrak{Q}$ is not $\varphi$-doubly commuting.

\noindent\textsf{Subcase (b): Assume $\theta_3(0) \ne 0$ and there exists $b'$ such that $b'\neq b$ and $\theta_3(b') = 0$.}  

Suppose, for the sake of contradiction, that $\mathfrak{Q}$ is $\varphi$-doubly commuting. Then from \eqref{theta1singu1} and \eqref{theta1singu2}, we have
$$
P_{\mathfrak{Q}}\left(\frac{bz S_1(z)}{w}\left[m_b(z,w)\theta_3(w) - \theta_3(0)\right]- z S_1(z) m_b(z,w)\theta_3(w) + z S_1(z) \theta_3(0) \right) = 0,
$$
which implies
$$
\frac{bz}{w}\left[m_b(z,w)\theta_3(w) - \theta_3(0)\right]-z m_b(z,w)\theta_3(w) + z\theta_3(0)=\theta_3(w) f_b(w) g(z,w)
$$
for some $ g \in \HD.$ At $w= b'$, the right-hand side vanishes as $\theta_3(b') = 0$, while the left-hand side becomes
$$\frac{(b'-b)\theta_3(0)z}{b'},$$
yielding a contradiction. Hence, $\mathfrak{Q}$ is not $\varphi$-doubly commuting.

\noindent\textsf{Case (iii): There exists a singular inner function $S_2$ such that $\theta_2(w)=w^nS_2(w),$ where $n \in \mathbb N.$ In addition, $S_2$ is nonconstant when $n=1.$} 

Consider the function $f(z, w)= w^{n-1}S_1(z)S_2(w).$ Then it evident that $f \in \mathfrak{Q}.$ We have $Q_2^*Q_1Q_2f(z,w)=0.$ On the other hand,
\beqn
Q_1Q_2Q_2^*f(z,w)&=&Q_1Q_2Q_2^*(w^{n-1}S_1(z)S_2(w))
\\
&=&\begin{cases}
Q_1\left(S_1(z)S_2(w)-S_1(z)S_2(0)\right)&\text{if}~n=1\\
Q_1(w^{n-1}S_1(z)S_2(w))&\text{if}~n \ge 2
\end{cases}
\\
&=&\begin{cases}
zS_1(z)S_2(w)-zS_1(z)S_2(0)&\text{if}~n=1\\
zw^{n-1}S_1(z)S_2(w)&\text{if}~n \ge 2.
\end{cases}
\eeqn
Hence, $\mathfrak{Q}$ is not $\varphi$-doubly commuting.

\noindent\textsf{Case (iv): For some $ b \ne 0$, $\theta_2(w) = f_b(w)^n S_2(w),$ where $S_2$ is a singular inner function, nonconstant when $n = 1.$}

Consider the function $f(z, w)= m_b(z, w)f_b(w)^{n-1}S_1(z)S_2(w).$ Now
\beq \label{theta1singu3}
Q_2^*Q_1Q_2f(z,w) \notag&=&Q_2^*Q_1Q_2\left(m_b(z, w)f_b(w)^{n-1}S_1(z)S_2(w)\right)\\
\notag&=&Q_2^*Q_1\left(bm_b(z, w)f_b(w)^{n-1}S_1(z)S_2(w)\right)\\
\notag&=& Q_2^*\left(zbm_b(z, w)f_b(w)^{n-1}S_1(z)S_2(w)\right)\\
&=&\frac{bzS_{1}(z)}{w}\left[m_b(z, w)f_b(w)^{n-1}S_2(w)-(-bc)^{n-1}S_2(0)\right].
\eeq
In the later equation \eqref{theta1singu3}, we have taken $f_{b}$ as in \eqref{automorphismdisc}, so $f_{b}(0)=(-bc)$. On the other hand
\beqn
Q_1Q_2Q_2^*f(z,w)&=&Q_1Q_2Q_2^*\left(m_b(z, w)f_b(w)^{n-1}S_1(z)S_2(w)\right)\\
&=&Q_1P_{\mathfrak{Q}}\left(m_b(z, w)f_b(w)^{n-1}S_1(z)S_2(w)-(-bc)^{n-1}S_1(z)S_2(0)\right)\\
&=&zm_b(z, w)f_b(w)^{n-1}S_1(z)S_2(w)-S_2(0)(-bc)^{n-1}Q_1P_{\mathfrak{Q}}\left(S_1(z)\right).
\\&=&zm_b(z, w)f_b(w)^{n-1}S_1(z)S_2(w)\\&-&S_2(0)(-bc)^{n-1}zS_1(z)\left[1-(-\bar b \bar c)^n(f_b(w))^nS_2(w)\overline{S_2(0)}\right].
\eeqn
This with \eqref{theta1singu3} yields that $\mathfrak{Q}$ is not $\varphi$-doubly commuting.  

\noindent\textsf{Case (v): $\theta_2(w)=f_b(w)$ where $b\neq 0$.} Take $f(z, w)=Q_1^*(S_1(z)-S_1(0)) \in \mathfrak{Q}.$ It is easy to see that $Q_1Q_2Q_2^*f=0$, whereas
\beqn
Q_2^*Q_2Q_1f(z,w)&=&Q_2^*Q_2Q_1Q_1^*(S_1(z)-S_1(0))\\
&=&Q_2^*Q_2(S_1(z)-S_1(0))\\
&=&Q_2^*P_{\mathfrak{Q}}(wS_1(z)-wS_1(0))\\
&=&S_1(z)-S_1(0)\neq 0
\eeqn
as $S_1$ is nonconstant. This completes the proof.
\end{proof}

We next provide a characterization of $\varphi$-doubly commuting quotient modules for inner functions of the form $\theta_1(z)\theta_2(w)$, under specific assumptions on their zeros.

\begin{theorem}\label{maincharasepa}
Let $\theta_1$ and $\theta_2$ be inner functions on $\mathbb D$ such that $\theta_1(a)=\theta_2(b)=0$ for some $a,b \in \mathbb D.$ Then $\mathfrak{Q}=\HD\ominus \theta_1(z)\theta_2(w)\HD$ is $\varphi$-doubly commuting if and only if $\theta_1(z) = c_1z$ and $\theta_2(w) = c_2w$ for some $c_1, c_2 \in \mathbb{T}.$
\end{theorem}

\begin{proof}
Let $\mathfrak{Q}$ be $\varphi$-doubly commuting. Lemmas \ref{Keylemma1} and \ref{keylemma2} together forces that $\theta_{1}(z)=c_{1}z$ for some $c_1 \in \mathbb T$. This, in turn, yields $\mathfrak{Q}=\HD \ominus c_{1} z \theta_{2}(w) \HD$. We now claim that $0$ is the only possible zero of $\theta_{2}$. Let, if possible, $\theta_{2}(0)\neq0$. Consider the function $f(z,w)=\theta_{2}(w)\in\mathfrak{Q}$. Then, $Q_{2}^*Q_{1}Q_{2}f(z,w)=0$, whereas
\[Q_{1}Q_{2}Q_{2}^*f(z,w)=Q_{1}(\theta_{2}(w)-\theta_{2}(0))=-P_{\mathfrak{Q}}(z\theta_{2}(0)) \neq 0\]
yielding a contradiction. Hence $\theta_{2}(w)=w^{m}\theta_{3}(w)$, where $\theta_{3}$ is an inner function. If $m > 1$, then this is not possible, because in that case $1, z, w, zw \in \mathfrak{Q}$. Therefore, $\theta_{2}(w)=w\theta_{3}(w).$ Consider the function $f(z,w)=\theta_{3}(w)$, which belongs to $\mathfrak{Q}$. Then $Q_{2}^*Q_{1}Q_{2}f(z,w)=0,$ while $Q_{1}Q_{2}Q_{2}^*f(z,w)=z\theta_{3}(w)-z\theta_{3}(0)$.
Since $\mathfrak{Q}$ is $\varphi$-doubly commuting, it follows that $\theta_3$ must be a constant inner function. The converse is immediate.
\end{proof}

We are now ready to prove Theorem \ref{finalcharasepa}.

\begin{proof}[Proof of Theorem \ref{finalcharasepa}]
Let $\mathfrak{Q}$ be $\varphi$-doubly commuting. 
\begin{enumerate}[{\normalfont(i)}]
\item If both $\theta_{1}$ and $\theta_{2}$ are singular, then according to Corollary \ref{keycoro1}, we deduce that both $\theta_{1}$ and $\theta_{2}$ are constant inner functions.
\item Let $\theta_{2}(b)=0$ for some $b\in\mathbb{D}$, and suppose that $\theta_{1}$ is a singular inner function. If $\theta_{1}$ is a nonconstant singular inner function, then by Lemma \ref{keylemma3}, we obtain a contradiction. Hence $\theta_{1}$ must be constant. This further implies that $\mathfrak{Q}=\HD\ominus\theta_{2}(w)\HD$, which is doubly commuting. Therefore, by Theorem \ref{DC+varDC}, we deduce that $\theta_{2}=f_{b}$ for some $b\in\mathbb{D}$ (see \eqref{automorphismdisc}).
\item Let $\theta_{1}(a) = 0$ for some $a \in \mathbb{D}$, and let $\theta_{2}$ be a singular inner function on $\mathbb{D}$. Then, by Lemmas \ref{Keylemma1} and \ref{keylemma2}, we have $\theta_{1}(z)=c_{1}z$ for some $c_{1}\in\mathbb{T}$. We now claim that $\theta_{2}$ must be constant. To see this, consider the function $f(z, w) = \theta_{2}(w)$, which belongs to $\mathfrak{Q}$. Then, $Q_{2}^*Q_{1}Q_{2}f(z,w)=0$, while 
\[Q_{1}Q_{2}Q_{2}^*f(z,w)=Q_{1}(\theta_{2}(w)-\theta_{2}(0))=-P_{\mathfrak{Q}}(z\theta_{2}(0))=z\theta_{2}(w)-z\theta_{2}(0).\]
Now the $\varphi$-doubly commutativity forces that $\theta_{2}$ is a constant function.
\item If $\theta_{1}$ and $\theta_{2}$ both have zeros, then by Theorem \ref{maincharasepa}, we obtain $\theta_{1}(z)=c_{1}z$ and $\theta_{2}(w)=c_{2}w$ for some $c_{1},c_{2}\in\mathbb{T}$.
\end{enumerate}
The converse part is immediate. This completes the proof.
\end{proof}

We now proceed to prove Theorem \ref{Hartogscharasepa}.

\begin{proof}[Proof of Theorem \ref{Hartogscharasepa}]
The proof follows from Theorem \ref{finalcharasepa}, along with \eqref{Equivalent_DC} and \eqref{Transference}.
\end{proof}

Finally, we conclude this section with a $\varphi$-doubly commuting quotient module $\HD \ominus\theta\HD$ such that there do not exist inner functions $\theta_{1}$ and $\theta_{2}$ on $\mathbb{D}$ satisfying
$$\theta(z,w)=\theta_{1}(z)\theta_{2}(w).$$

\begin{proposition} \label{notsepatheta}
Let $\theta_a(z, w) = \frac{zw - a}{1 - azw}$, where $0<a<1$. Then, the quotient module $\mathfrak{Q}_{a}=\HD\ominus\theta_{a}\HD$ is $\varphi$-doubly commuting.
\end{proposition}

\begin{proof}
It follows from \cite[p. 3225]{GW2009} that $\theta_a$ is not separating for any $a \in (0,1).$ By \cite[Lemma 3.3]{GW2009}, the set  
$$B_a=\left\{(1 + a\theta_a)w^{j}, (1 + a\theta_a), (1 + a\theta_a)z^{i}:i,j\ge1\right\}$$
forms an orthogonal basis for $\mathfrak{Q}_{a}$. Let $P_a:\HD\to\mathfrak{Q}_{a}$ denote the orthogonal projection. We observe that for each $i\ge1$
\begin{align}
\nonumber Q_{2}(1+a\theta_{a})z^{i}&=P_{a}\left[(1+a\theta_{a})z^{i}w\right]\\
\nonumber&=P_{a}\left[z^{i-1}(1+a\theta_{a})zw\right]\\
\nonumber&=P_{a}\left[z^{i-1}(1+a\theta_{a})((1-azw)\theta_a +a)\right]\\
\label{theta_a-1}&=az^{i-1}(1+a\theta_{a}).
\end{align}
The last equality holds because $z^{i-1}(1+a\theta_{a})(1-azw)\theta_a\in\theta_{a}\HD$. We also obtain
\begin{align}\label{theta_a-2}
Q_2(1+a\theta_a)w^j=(1+a\theta_a)w^{j+1}, \quad j\ge0.
\end{align}
Similarly, we obtain
\begin{align}\label{theta_a-3}
Q_{1}(1+a\theta_a)z^i=(1+a\theta_a)z^{i+1}, \quad i\ge0,
\end{align}
and
\begin{align}\label{theta_a-4}
Q_{1}(1+a\theta_{a})w^{j}=P_{a}\left[w^{j-1}(1+a\theta_{a})zw\right]=aw^{j-1}(1+a\theta_{a}),\quad \mbox{for all}~j\ge1.
\end{align}
Let $i\ge0$ be fixed. Clearly, for $j\ge0$
\begin{align}\label{theta_a-5}
\langle Q_2^*(1+a\theta_a)z^i,  (1+a\theta_a)w^j\rangle=\langle (1+a\theta_a)z^i,(1+a\theta_a)w^{j+1}\rangle=0.
\end{align}
Further, for $j\ge1$ we calculate
\begin{align}\label{theta_a-6}
\nonumber\langle Q_2^*(1+a\theta_a)z^i, (1+a\theta_a)z^j\rangle &=\langle (1+a\theta_a)z^i, Q_2\left[(1+a\theta_a)z^j\right]\rangle \\ \nonumber&\overset{\eqref{theta_a-1}}{=}\langle (1+a\theta_a)z^{i},  az^{j-1}(1+a\theta_a)\rangle\\
&=\begin{cases}
a\|(1+a\theta_a)z^{i}\|^2 & \text{ if }~j=i+1,\\
0 &\text{ elsewhere}.
\end{cases}
\end{align}
Thus, by \eqref{theta_a-5} and \eqref{theta_a-6} we have
\begin{align}\label{theta_a-8}
Q_2^*(1+a\theta_a)z^{i}=a\|(1+a\theta_a)z^{i}\|^2(1+a\theta_a)z^{i+1}\quad\mbox{for $i\ge 0$.}
\end{align}
Using \eqref{theta_a-1} and \eqref{theta_a-2} we similarly calculate
\begin{align}\label{theta_a-9}
Q_2^*(1+a\theta_a)w^j=\|(1 +a\theta_a)w^j\|^2(1+a\theta_a)w^{j-1}\quad\mbox{for $j \ge 1$}.
\end{align}
Now, from \eqref{theta_a-1}-\eqref{theta_a-4}, \eqref{theta_a-8} and \eqref{theta_a-9} we finally have
$$[Q_{1}Q_{2},Q_{2}^*]f=0\quad\mbox{for all $f\in B_{a}$}.$$
This completes the proof.
\end{proof}

\section{Essential normality and doubly commutativity of quotient modules induced by polynomials}\label{Sec4}
In this section, we study the quotient modules of the form $\mathcal{Q}_p:=\HT\ominus p\HT$, where $p\in\mathbb{C}[z,w].$ For each $m \in \mathbb Z_+,$ let $\mathfrak F_m$ denote the finite dimensional closed subspace of $\HT$ spanned by $F_m,$ where 
$$F_m=\Big\{\frac{1}{w}\left(\frac{z}{w}\right)^{m-j}w^j\Big\}_{j=0}^{m}.$$ 
It follows that $$\HT=\bigoplus_{m=0}^{\infty}\mathfrak F_m.$$ 

\begin{definition}
Given a monomial $z^i w^j$ with nonnegative integer exponents $i,j$, we define its \textit{Hartogs degree} as $ 2i+j+1 $.  
\end{definition}

It is worth noting that two monomials of different usual degrees may have the same Hartogs degree. For instance, consider the monomials $p(z, w)=z$ and $q(z, w)=w^2.$ While their degrees differ, their Hartogs degree is the same, both being equal to $3.$

\begin{remark}
 A function $f\in\mathcal{Q}_p\cap \mathfrak{F}_m$ given by $f(z,w)=\sum_{i=0}^{m}c_{i}\frac{1}{w}\left(\frac{z}{w}\right)^{m-i}w^{i}$ for some scalars $c_0, c_1,\ldots,c_m$ if and only if 
\beq \label{Inner-product expansion}
\sum_{i=0}^{m} c_i \left\langle \frac{1}{w} \left(\frac{z}{w}\right)^{m-i} w^i, p(z, w) g(z, w) \right\rangle=0, \quad g \in \HT, m \in \mathbb Z_+.
\eeq \eof
\end{remark}

\subsection{Dimensional analysis} \label{Dim analysis}
In this subsection, we are interested in describing the dimension of $\mathcal{Q}_{p}\cap\mathfrak{F}_{m}$ in terms of $p$ and $m.$ This study will provide the necessary foundation for establishing the essential normality and the doubly commutativity of $\mathcal{Q}_{p}$. 

\begin{theorem} \label{dimensional analysis}
Let $m \in \mathbb Z_+.$ Then the following statements are valid.
\begin{itemize}
\item [$\mathrm {(i)}$] Let $p(z, w) =z^qw^n,$ where $q ,n \in \mathbb Z_+.$ Then 
$$\operatorname{dim}(\mathcal{Q}_p \cap \mathfrak{F}_m)= \begin{cases} 
m+1 & \text{if } m <2q+n, \\
2q+n & \text{if } m \ge 2q+n. 
\end{cases}$$
\item [$\mathrm {(ii)}$] Let $\displaystyle p(z,w) =\sum_{i+j \geq 1} a_{i,j} z^i w^j$ be a polynomial in which no two summands have the same Hartogs degree. Then there exists $m_0 \in \mathbb N$ such that 
$$\operatorname{dim}(\mathcal{Q}_p \cap \mathfrak{F}_m) = \min\{i+j : a_{i, j} \neq 0 \text{ in } p\} + \min\{i : a_{i, j} \neq 0 \text{ in } p\}$$ for all $m \ge m_0.$
\item [$\mathrm {(iii)}$] Let $b_k \in \mathbb{C} \setminus \{0\}$ for $0 \leq k \leq t$ for some $t \in \mathbb Z_+.$ If $\displaystyle p(z,w)=\sum_{k=0}^t b_k z^{q_k}w^{n_k}$, where $q_{k}\le q_{k+1}$ for all $0\le k\le t-1$, is a polynomial with each summand having same Hartogs degree, then 
$$\operatorname{dim}(\mathcal{Q}_p \cap \mathfrak{F}_m)= \begin{cases} 
m+1 & \text{if } m <2q_{0}+n_{0}, \\
2q_{0}+n_{0} & \text{if } m\ge q_{t}+q_{0}+n_{0}.
\end{cases}$$
\end{itemize}   
\end{theorem}

\begin{proof}
$\mathrm{(i)}$ Note that $p$ has the Hartogs degree $ 2q+n+1 $ and 
\beqn
p(z, w) = z^qw^n = \frac{1}{w} \left(\frac{z}{w}\right)^{q} w^{q+n+1}, \quad q, n \in \mathbb Z_+.
\eeqn 
	
For $m<2q+n$, it is easy to see that $\mathfrak{F}_m\subseteq \mathcal{Q}_p$. This yields that  
\beqn
\operatorname{dim} (\mathcal{Q}_p \cap \mathfrak{F}_m)=m+1.
\eeqn  
Assume that $ m \geq 2q+n$. Let $f(z, w)=\sum_{j=0}^{m} c_j \frac{1}{w} \left(\frac{z}{w}\right)^{m-j} w^j \in \mathfrak{F}_m\cap\mathcal{Q}_p.$ For $ q+n \leq j \leq m-q $, the function $ g(z, w)=g_j(z, w) = \left(\frac{z}{w}\right)^{m-q-j} w^{j-q-n-1} $ and $p(z, w)=z^qw^n$ in \eqref{Inner-product expansion} yields that $c_j=0$ (since $f \in \mathfrak{F}_m\cap\mathcal{Q}_p$). Therefore, the dimension of $\mathcal{Q}_p\cap\mathfrak{F}_m$ is at most $2q+n$. Note that the smallest possible power of $\frac{z}{w}$ and $w$ in any summand of $$z^qw^ng(z,w)=\frac{1}{w}\left(\frac{z}{w}\right)^{q}w^{q+n+1}g(z,w)$$ 
is $q$ and $q+n$, respectively, for any $g\in\HT$. In contrast, for $j>m-q$, the largest possible power of $\frac{z}{w}$ in any summand of $f(z,w)$ is $q-1.$ Whereas for $j<q+n,$ the largest possible power of $w$ in any summand of $f(z,w)$ is $q+n-1$. This yields that $$ \frac{1}{w} \left(\frac{z}{w}\right)^{m-j} w^j\in \mathcal{Q}_p \cap \mathfrak{F}_m \mbox{ for all } j=0, \ldots, q+n-1, m-q+1,\ldots, m.$$
Therefore, the dimension of $\mathcal{Q}_p\cap \mathfrak{F}_m$ is exactly $2q+n$.
	
$\mathrm{(ii)}$ Let $i_1$ and $j_1$ be integers such that 
\beqn
i_1+j_1=\min\{i+j: a_{i, j} \neq 0 \text{ in } p\}.
\eeqn  
Choose $m \ge 2i_1+j_1.$ Let $f(z, w)=\sum_{k=0}^{m} c_k \frac{1}{w} \left(\frac{z}{w}\right)^{m-k} w^k \in \mathfrak{F}_m\cap\mathcal{Q}_p.$ Then for $ i_1 + j_1 \leq k \leq m - i_1 ,$ choosing $g(z,w)=g_k(z, w)= \left(\frac{z}{w}\right)^{m-i_1-k} w^{k-i_1-j_1-1}$ and $p(z,w)=\sum_{i+j \geq 1} a_{i,j} z^i w^j$ in \eqref{Inner-product expansion}, we obtain $c_{k}=0$ by using the fact no two summands in $p$ have same Hartogs degree.
	
Furthermore, there exists $ i_2 \geq 0 $ such that  
\beqn
i_2 = \min\{i : a_{i, j} \neq 0 \text{ in } p\}.
\eeqn  
Clearly, $ i_2 \leq i_1 $, and it follows that $ c_k = 0 $ for $m-i_1 \le k \le m-i_2$ for a suitable choice of $g$ in \eqref{Inner-product expansion} and for sufficiently large $m$. The conclusion can now be drawn using an argument analogous to that in $\mathrm{(i)}.$
	
$\mathrm{(iii)}$ Let $$p(z,w)=\sum_{k=0}^t b_k z^{q_k}w^{n_k}$$ be a polynomial such that $2q_k+n_k=a$ for all $0 \le k \le t$ and for some $a \in \mathbb N.$ It is easy to see that $\mathfrak{F}_m \subseteq \mathcal{Q}_p$ for $m < a.$ We claim that the dimension of $ \mathcal{Q}_p \cap \mathfrak{F}_m$ is $2q_0+n_0$ when $m \ge q_t+q_0+n_0$. Since $q_k\le q_{k+1}$ for all $0\le k\le t-1,$ we have $n_{k+1}\le n_{k}$ and $q_{k+1}+n_{k+1}\le q_k +n_k$ for all $0 \le k \le t-1.$ Then for any $g \in \HT$ and $0 \le i \le q_t+n_t-1$
\beq \label{Initial independence-1}
\left\langle \frac{1}{w} \left(\frac{z}{w}\right)^{m-i}w^i, p(z, w) g(z, w) \right\rangle=0.
\eeq
Let $f(z, w)=\sum_{k=q_t+n_t}^{m} a_k \frac{1}{w} \left(\frac{z}{w}\right)^{m-k} w^k \in \mathfrak{F}_m\cap\mathcal{Q}_p.$ For $ q_t+n_t \leq j \leq m-q_t $, the function $ g(z, w)=g_j(z, w) = \left(\frac{z}{w}\right)^{m-q_t-j} w^{j-q_t-n_t-1} $ and $p(z,w)=\sum_{k=0}^t b_{k}z^{q_k}w^{n_k}$ in \eqref{Inner-product expansion} gives
\beq \label{dependent equation-1}
\sum_{k=0}^{t} \overline{b_{t-k}}\,a_{q_{t}-q_{t-k}+j} = 0 \quad\text{for}\quad q_t+n_t \leq j \leq m-q_t,
\eeq
where we used the fact that $q_k+n_k-q_t-n_t=q_t-q_k$ for all $k=0, \ldots,t.$ Let $0\le l\le t-1$. Additionally, for $q_{l}+n_{l}\le j\le m-q_{l}$ considering $g(z,w)=g_{j}^{l}(z,w)=\left(\frac{z}{w}\right)^{m-q_{l}-j} w^{j-q_{l}-n_{l}-1}$ in \eqref{Inner-product expansion} further gives
\beq\label{dependent equation-2}
\sum_{k=0}^{t} \overline{b_{t-k}}\,a_{q_{l}-q_{t-k}+j} = 0 \quad\text{for}\quad q_{l}+n_{l}\le j\le m-q_{l},
\eeq
Since $2q_{l}+n_{l}=a$ for all $0\le l\le t$, it follows that the system of equations appearing in \eqref{dependent equation-2} is already present in  \eqref{dependent equation-1} for each $0\le l\le t$. Therefore, it is enough to consider only \eqref{dependent equation-1}. Note that the number of unknowns in \eqref{dependent equation-1} is $m+1-q_{t}-q_0-n_t$. We want to note here that the condition $m\ge q_{t}+q_{0}+n_{0}$ justifies the last argument. If $A$ is the coefficient matrix of the homogeneous system \eqref{dependent equation-1}, then $A$ is already in echelon form of rows. Furthermore, since $b_{k}\neq0$ for all $0\le k\le t$, we have $m+1-2q_t-n_t$ pivot entries in $A$. Hence, the dimension of the solution space will be $q_t-q_0.$ 
	
Finally, for $m-q_{0}< j\le m$ and $g \in \HT,$ 
\beq\label{final independence-1}
\left\langle \frac{1}{w} \left(\frac{z}{w}\right)^{m-j}w^j, p(z, w) g(z, w) \right\rangle=0. 
\eeq
Thus, from \eqref{Initial independence-1}, \eqref{dependent equation-1} and \eqref{final independence-1}, we conclude that dimension of $ \mathcal{Q}_p \cap \mathfrak{F}_m$ is $2q_0+n_0$. This proves the claim.
\end{proof}
Here is an immediate corollary.

\begin{corollary}
The following statements are valid:
\begin{itemize}
\item[$\mathrm{(i)}$] If $p(z)=\sum_{i =1}^{k}a_iz^i$ is a polynomial then 
$$\operatorname{dim}(\mathcal{Q}_p \cap \mathfrak{F}_m) = 2\min\{i : a_{i} \neq 0 \text{ in } p\}$$ for $m \ge 2\min\{i : a_{i} \neq 0 \text{ in } p\}.$
\item[$\mathrm{(ii)}$] If $p(w)=\sum_{i =1}^{k}a_iw^i$ is a polynomial then 
$$\operatorname{dim}(\mathcal{Q}_p\cap \mathfrak{F}_m)=\min\{i:a_{i}\neq 0 \text{ in } p\}$$ for $m \ge \min\{i : a_{i} \neq 0 \text{ in } p\}.$
\item [$\mathrm{(iii)}$] Let $p(z, w) =\sum_{i =1}^{k}a_iz^i + \sum_{j=1}^{l}b_jw^j$ be a polynomial with each summand having a different Hartogs degree and $a_ib_j \neq 0$ for some $i$ and $j.$ Then
$$\operatorname{dim}(\mathcal{Q}_p \cap \mathfrak{F}_m )=\min \{\min\{i : a_{i} \neq 0 \text{ in } p \}, \min\{j : b_{j} \neq 0  \text{ in } p\}\}$$
for $m\ge 2\min\{i : a_{i} \neq 0 \text{ in } p \}+\min\{j : b_{j} \neq 0  \text{ in } p\}$.
\end{itemize}
\end{corollary}

It is natural to consider a general polynomial of the form $ p(z,w) = \sum_{j=1}^{k} p_j(z, w) $, where each summand $ p_j $ is a polynomial consisting terms with the same Hartogs degree for every $ 1 \leq j \leq k $. However, obtaining an explicit formula, as in the previous cases, appears to be challenging. The following example illustrates this difficulty.

\begin{example}
Consider the polynomial \beq \label{Exampoly1}
p(z,w)=zw^5+z^2w^3+z^3w^5+z^5w\eeq with components $p_1(z,w)=zw^5+z^2w^3$
 and $p_2(z,w)=z^3w^5+z^5w$. The Hartogs degrees of $p_1$ and $p_2$ are $8$ and $12$, respectively.
	
Assume that $m \ge 30.$ Let $f(z, w)=\sum_{j=0}^{m} a_j \frac{1}{w} \left(\frac{z}{w}\right)^{m-j} w^j \in \mathfrak{F}_m\cap\mathcal{Q}_p.$ The function
\beqn g_j(z, w) = \left(\frac{z}{w}\right)^{m-2-j} w^{j-6} \eeqn
in \eqref{Inner-product expansion} with $p$ as given in \eqref{Exampoly1} leads to the recurrence relation
\beq \label{1}	
a_j + a_{j+1} = 0, \quad 5 \le j \le m-2.
\eeq
Similarly, the function
\beqn g_j(z, w) = \left(\frac{z}{w}\right)^{m-5-j} w^{j-7} \eeqn
in \eqref{Inner-product expansion} with $p$ as given in \eqref{Exampoly1} yields
\beq \label{2}
a_j + a_{j+2} = 0, \quad 6 \le j \le m-5.
\eeq
From \eqref{1}, we obtain $a_j = (-1)^{j+1} a_5$ for $5\le j\le m-1$. Setting $j=6$ in \eqref{2} gives $a_5=0$, leading to the dimension of $\mathcal{Q}_p\cap\mathfrak{F}_m$ being $6$.
	
Next, consider $q(z,w)=q_1(z, w)+q_2(z, w)$, where
$q_1(z, w)=zw^5+z^2w^3$ and $q_2(z, w) = z^5w^6 + z^8.$ Here, the Hartogs degrees of $q_1$ and $q_2$ are $8$ and $17$, respectively.
	
Assume that $m \ge 30.$ Let $f(z, w)=\sum_{j=0}^{m} a_j \frac{1}{w} \left(\frac{z}{w}\right)^{m-j} w^j \in \mathfrak{F}_m\cap\mathcal{Q}_q.$ Choosing the functions
\beqn g_j(z,w)=\left(\frac{z}{w}\right)^{m-2-j} w^{j-6}\quad\text{for}\quad5\le j\le m-2
\eeqn
and 
\beqn 
g_j(z, w) = \left(\frac{z}{w}\right)^{m-8-j} w^{j-9}\quad\text{for}\quad8\le j\le m-8 
\eeqn
in \eqref{Inner-product expansion} we have the relations
\beq \label{3}
a_j+a_{j+1}=0, \quad 5 \leq j \leq m-2
\eeq
and
\beq \label{4}
a_j + a_{j+3} = 0, \quad 8 \leq j \leq m-8,
\eeq
respectively. Finally, these relations, \eqref{3} and \eqref{4}, together imply that the dimension of $\mathcal{Q}_q\cap\mathfrak{F}_m$ is $7$.\eoex
\end{example}

\subsection{Essential normality of $\mathcal{Q}_p$}

\cite[Theorem 1.1]{GuWa2007} provides a characterization of essentially normal quotient modules in $\HD$ for a certain class of polynomials. Here, we present a characterization of essentially normal quotient modules $\mathcal{Q}_p$ of $\HT$, under mild assumptions on the polynomial $p\in\mathbb{C}[z,w]$. In fact, we consider those polynomials $p$ for which
\begin{align}\label{Class-p}
\mathcal Q_p=\bigoplus_{m=0}^{\infty}(\mathfrak F_m\cap \mathcal Q_p).
\end{align}
For $m\in\mathbb{Z}_+,$ let 
\begin{align}\label{Shortcut-Nota}
E_{m-j}^m:=\frac{1}{w}\left(\frac{z}{w}\right)^{m-j}w^j, \quad j=0,1, \ldots, m. 
\end{align} 

We begin with the following elementary observation.
\begin{lemma}\label{Not essentially normal lemma}
Let $p$ be a polynomial and $l\in\mathbb{N}$ such that $E^k_k$ and $E^k_{k-1}$ are in $\mathcal{Q}_{p}$ for all $k\ge l$ $\rm{(}$see \eqref{Shortcut-Nota}$\rm{)}$. Then $\mathcal{Q}_p $ is not essentially normal.
\end{lemma}

\begin{proof}
Let $k\ge l$. First, note that  
\beqn
Q_z E^k_k=P_{\mathcal{Q}_{p}}M_{z} \left[\frac{1}{w} \left(\frac{z}{w}\right)^{k}\right]=P_{\mathcal{Q}_{p}}\left[\frac{1}{w} \left(\frac{z}{w}\right)^{k+1} w \right]= E^{k+2}_{k+1}.
\eeqn  
Similarly, 
\beqn
Q_z^* E^k_k=P_{\mathcal{Q}_{p}} M_{z}^* \left[\frac{1}{w} \left(\frac{z}{w}\right)^{k} \right]= 0.
\eeqn  
Also,  
\beqn
Q_z^* E^{k+2}_{k+1}=P_{\mathcal{Q}_{p}}M_{z}^* \left[\frac{1}{w} \left(\frac{z}{w}\right)^{k+1} w \right]= E^{k}_{k}.
\eeqn  
It follows that  
\beq \label{Q_z not compact}
[Q_z^*,Q_z] E^k_k=Q_z^* Q_z E^k_k-Q_z Q_z^* E^k_k= Q_z^*E^{k+2}_{k+1}=E^k_k.
\eeq 
Since $ E^k_k \neq 0 $ for all $ k \ge l$ and linearly independent, $[Q_z^*,Q_z]$ is not a compact operator. Note that $Q_wE_k^k=E^{k+1}_k,$ $Q_w^*E_k^k=0.$ It is easy to see that $$(Q_w^*Q_z-Q_zQ_w^*)E_k^k=Q_w^*E^{k+2}_{k+1}-0=E^{k+1}_{k+1}.$$ 
Hence, $[Q_w^*, Q_z]$ is also not compact. Therefore, $\mathcal{Q}_p$ is not essentially normal.
\end{proof}

\begin{remark}
From the proof of Theorem \ref{dimensional analysis}, it follows that for monomials of the form $p(z,w)=z^n$ (for $n\ge2$), $p(z,w)=w^m$ (for $m\ge2$), or  $p(z,w)=z^m w^n$ (for $m,n\ge 1 $), we observe that $E^k_k,E^k_{k-1}\in\mathcal{Q}_p$ for $k \ge 1$. Hence by Lemma \ref{Not essentially normal lemma}, $\mathcal{Q}_p $ is not essentially normal. \eof
\end{remark}
As a consequence, we get the following.
\begin{proposition} \label{a=b=0case}
Let $p$ be a polynomial in the variables $z$ and $w$ represented as 
\beq \label{thepolynomial}
p(z,w)=a z+ b w + cw^2 +w^3q(w) +z^2r(z)+\sum_{i,j \ge 1} a_{i,j} z^i w^j,
\eeq
where $a, b, c$ are scalars, and $q(w)$ and $r(z)$ are single variable polynomials. If $a=b=0$, then the quotient module $\mathcal{Q}_p$ is not essentially normal.
\end{proposition}

\begin{proof}
If $a=b=0$, then the polynomial $p$ given in \eqref{thepolynomial} can be represented as 
\beqn
p(z,w)=\sum_{i+j \geq 2} a_{i,j} z^i w^j.
\eeqn  
It is easy to see that for $k\ge 1$,
\beqn
\langle E^k_k, p(z,w) g(z, w) \rangle = \langle E^k_{k-1}, p(z,w) g(z, w) \rangle = 0 \quad \forall~g \in \HT.
\eeqn  
Hence, by Lemma \ref{Not essentially normal lemma}, $ \mathcal{Q}_p $ is not essentially normal.
\end{proof}

We now turn to the case of quotient modules $\mathcal{Q}_{p}$ that are essentially normal.

\begin{theorem} \label{Essential normal theorem}
Let $p$ be a polynomial in the variables $z$ and $w$ given by \eqref{thepolynomial} such that $\mathcal{Q}_{p}$ satisfies \eqref{Class-p}. Then the following statements are valid.
\begin{itemize}
\item [$\mathrm{(i)}$] If $b=1$ then the quotient module $\mathcal Q_p$ is essentially normal.
\item [$\mathrm{(ii)}$] If $a=1, b=c=0$ then the quotient module $\mathcal Q_p$ is essentially normal.
\item [$\mathrm{(iii)}$] If $a=c=1, b=0$ then the quotient module $\mathcal Q_p$ is essentially normal.
\end{itemize}
\end{theorem}

\begin{proof}
$\mathrm{(i)}$ Let $b=1$ in $\eqref{thepolynomial}.$ Then the polynomial $p$ takes the form 
\beq\label{Case:b=1}
p(z,w) = az+w+\sum_{i+j \ge 2} a_{i,j} z^i w^j.
\eeq
Note that the Hartogs degree of $w$ will never be equal to the Hartogs degree of any summand in $az +\sum_{i+j \ge 2}a_{i,j}z^i w^j$. Let $f(z, w)=\sum_{k=0}^{m} a_k \frac{1}{w} \left(\frac{z}{w}\right)^{m-k} w^k \in \mathcal{Q}_p\cap \mathfrak{F}_m.$ For $1\le k\le m $, selecting $ g_k(z, w) =\left(\frac{z}{w}\right)^{m-k} w^{k-2}$ and $p(z,w)$ (as given in \eqref{Case:b=1}) in \eqref{Inner-product expansion} yields that $a_k=0$.  
On the other hand, for any $g \in \HT,$
\beqn
\left\langle \frac{1}{w} \left(\frac{z}{w}\right)^{m}, \left(az+w+\sum_{i+j \geq 2} a_{i,j} z^i w^j\right) g(z, w) \right\rangle=0.
\eeqn
This further implies 
$$\mathcal{Q}_p\cap \mathfrak{F}_m=\overline{\text{span}}\left\{\frac{1}{w} \left(\frac{z}{w}\right)^{m}\right\}=\overline{\text{span}}\left\{E_m^m\right\}\quad\text{for}~m\ge0.$$

It is now easy to see that $\mathcal{Q}_p$ is essentially normal.
	
$\mathrm{(ii)}$ Let $a=1, b=c=0$ in $\eqref{thepolynomial}.$ Then the polynomial $p$ takes the form
\beq\label{Case:a=1,b=c=0}
p(z,w)=z+w^3q(w)+z^2r(z)+\sum_{i,j\ge1}a_{i,j}z^{i}w^{j}.
\eeq
For all $m\ge0$ and for any $g \in \HT,$ we have
\beqn
\left\langle E^m_m, p(z, w)g(z, w) \right\rangle=0,
\eeqn 
it follows that $E^{m}_{m}\in\mathcal{Q}_{p}\cap\mathfrak{F}_{m}$ for all $m\ge0$. Moreover, $E^{1}_{0}\in\mathcal{Q}_{p}\cap\mathfrak{F}_{1}$. We observe that the Hartogs degree of $z$ will never be equal to the Hartogs degree of any summand in $w^3q(w)+z^2r(z)+\sum_{i,j \ge 1} a_{i,j} z^i w^j.$ Let $f(z, w)=\sum_{k=0}^{m} a_k \frac{1}{w} \left(\frac{z}{w}\right)^{m-k} w^k \in \mathcal{Q}_p\cap \mathfrak{F}_m$ and $m\ge2$. For $ 1 \leq k \leq m-1 $, choosing $g_k(z, w)=\left(\frac{z}{w}\right)^{m-k-1} w^{k-2} $ and $p(z,w)$ (as in \eqref{Case:a=1,b=c=0}) in \eqref{Inner-product expansion} we obtain $a_k=0$. It now remains to verify whether $E_{0}^{m}\in\mathcal{Q}_{p}\cap\mathfrak{F}_{m}$ or not for $m \ge 2$. Let $k\in\mathbb{N}\cup\{0\}$ such that $$q(w)=\sum_{j=k}^{\operatorname{deg}q}c_jw^j,$$ 
where $c_k\neq0.$ Therefore, if $m<k+3$, we obtain
\beqn
\left\langle E^m_0, p(z, w) g(z, w) \right\rangle=0
\eeqn 
for all $g\in\HT$. For $m\ge k+3$, we obtain $a_{m}=0.$ Consequently
\begin{align}\label{QpFm-ii}
\mathcal{Q}_p \cap \mathfrak{F}_m=\begin{cases}
\overline{\text{span}}\left\{E^m_m,E^{m}_{0}\right\} &\text{for}~0\le m<k+3,\\
\overline{\text{span}}\left\{E^m_m\right\} &\text{for}~m\ge k+3.
\end{cases}    
\end{align}
It is now easy to see that the corresponding quotient module $\mathcal{Q}_p$ is essentially normal.

$\mathrm{(iii)}$ Let $a=c=1, b=0.$ Then the polynomial $p$ given in \eqref{thepolynomial} takes the form 
\begin{align}\label{Case:a=c=1,b=0}
p(z,w)=z+ w^2 +w^3q(w) +z^2r(z)+\sum_{i,j \ge 1} a_{i,j} z^i w^j.
\end{align}
Note that the Hartogs degree of $z$ and $w^2$ is $3.$ However, the Hartogs degree of any summand in $w^3q(w) +z^2r(z)+\sum_{i,j \ge 1} a_{i,j} z^i w^j$ is at least $4$.

For $m\ge0$ and for any $g \in \HT,$ we have
\beqn
\left\langle E^m_m, p(z, w)g(z, w) \right\rangle=0.
\eeqn 
Thus $E^{m}_{m}\in\mathcal{Q}_{p}\cap\mathfrak{F}_{m}$ for all $m\ge0$. Moreover, $E^{1}_{0}\in\mathcal{Q}_{p}\cap\mathfrak{F}_{1}$.

Let $m\ge2$, and suppose 
\begin{align*}
f(z,w)=\sum_{i=1}^{m}a_{i}\frac{1}{w}\left(\frac{z}{w}\right)^{m-i}w^{i}\in\mathcal{Q}_{p}\cap\mathfrak{F}_{m}.
\end{align*}
For $1\leq k\leq m-1$, take $g_k(z,w)=\left(\frac{z}{w}\right)^{m-k-1}w^{k-2}$ and $p$ (as in \eqref{Case:a=c=1,b=0}) in \eqref{Inner-product expansion}, We obtain 
\begin{align*}
a_k+a_{k+1}=0,
\end{align*} 
which further yields that $a_k=(-1)^{k+1}a_1$ for $k=1,\ldots,m$. 

Let $k\in\mathbb{N}\cup\{0\}$ such that $q(w)=\sum_{j=k}^{\operatorname{deg}q}c_jw^j,$ where $c_k\neq0.$ Then $a_{m}=0$ for $m\ge k+3$, which implies
$$\mathcal{Q}_p \cap \mathfrak{F}_m=\overline{\text{span}}\left\{E^m_m\right\}\quad\text{for}~m\ge k+3.$$
For $0\le m< k+3$, it is possible that $a_{1}=0$ or $a_{1}\neq0,$ depending on the choice of the polynomial $r(z)$ and the coefficient $a_{i,j}$. Therefore, 
$$\mathcal{Q}_{p}\cap\mathfrak{F}_{m}\subseteq\overline{\text{span}}\Big\{E_{m}^{m},a_{1}\sum_{k=1}^{m}(-1)^{k+1} E_{m-k}^m\Big\}\quad\text{for}~0\le m< k+3.$$
It is now easy to see that the corresponding quotient module $\mathcal{Q}_p$ is essentially normal.
This completes the proof.
\end{proof}

As a consequence of Proposition \ref{a=b=0case} and Theorem \ref{Essential normal theorem}, we establish Theorem \ref{Not essentially normal Theorem}.

\begin{proof}[Proof of Theorem \ref{Not essentially normal Theorem}]
Let $p$ be as given in \eqref{thepolynomial}. Then it follows from Theorem \ref{Essential normal theorem} and Proposition \ref{a=b=0case} that $\mathcal{Q}_p$ is not essentially normal if and only if $a=b=0$.
\end{proof}
\begin{remark}
Interestingly, when $p(z,w)=(z-w)^2$, the quotient module $H^{2}(\mathbb{D}^{2})\ominus pH^{2}(\mathbb{D}^{2})$ is essentially normal in the Hardy space of the bidisc (see \cite[Section~5]{DM1993}). However, in the case of the Hardy space of the Hartogs triangle, $\HT\ominus p\HT$ is not. This contrast highlights a key difference between these two function spaces.\eof
\end{remark}

We now present a few examples where equation \eqref{Class-p} fails, as well as examples where it holds.

\begin{example}
\begin{enumerate}[{\normalfont(i)}]
\item Consider the polynomials, $z^{m},w^{n},z^{m}w^{n}$ and $(z-w)^2,$ where $m,n\in\mathbb{N}$. Then $\mathcal{Q}_{p}$ satisfies \eqref{Class-p}.
\item Consider the polynomial $p(z,w)=w-\frac{1}{2}$. Observe that the Hartogs degree of the constant term is $1$, which differs from the Hartogs degree of $w$. This, in turn, implies that $\mathcal{Q}_{p}\cap\mathfrak{F}_{m}=\{0\}$ for all $m\ge0.$ However, $\sum_{i=0}^{\infty}\frac{1}{2^{i}}\frac{1}{w}w^{i}\in\mathcal{Q}_{p}.$\eoex
\end{enumerate}
\end{example}

We end this section with an application of Theorem \ref{Essential normal theorem}. Recall that a quotient module $\HD\ominus p\HD$ is doubly commuting for a polynomial $p$ if and only if $p\HD=G\HD$ for some finite Blaschke product $G$ depending on a single variable (see \cite[Corollary 4.3]{DY2000}). In contrast, in $\HT$, there exist doubly commuting quotient modules $\mathcal{Q}_{p}=\HT\ominus p\HT$ for which $p$ is not necessarily an inner function.

\begin{proposition}\label{dcpolycase}
Let $p(z,w)=a z+ b w + cw^2 +w^3q(w) +z^2r(z)+\sum_{i,j \ge 1} a_{i,j} z^i w^j$ such that $\mathcal{Q}_p$ satisfies \eqref{Class-p}. If either $b \neq 0$ or $a \neq0$ and $b=c=0$, then the quotient module $\mathcal{Q}_p$ is doubly commuting.
\end{proposition}

\begin{proof}
Consider the case when $b\neq0$. Then, by Theorem \ref{Essential normal theorem} (i), we have $\mathcal{Q}_{p}\cap\mathfrak{F}_{m}=\overline{\text{span}}\{E_{m}^{m}\}$ for $m\ge0$. This trivially implies the doubly commutativity. 
	
Now consider the case $b=0$, $a\neq0$ and $c=0$, then by Theorem \ref{Essential normal theorem} (ii) and \eqref{QpFm-ii}, we again have the doubly commutativity.
\end{proof}

\section{Concluding remarks}
In this article, we introduced the concept of $\varphi$-doubly commuting quotient modules of $\HD$ (see Definition \ref{phi-dc}) to classify the doubly commuting quotient modules of $\HT$ (see Theorem \ref{Hartogscharasepa}). Although the present approach does not provide a complete classification of doubly commuting quotient modules of $\HT$, it establishes a connection between the doubly commuting quotient modules of $\HD$ and those of $\HT$.
 
We wrap up this paper with the following natural question for future investigation:

\begin{center}
\emph{Characterize all doubly commuting quotient modules of $\HT$.}
\end{center}

\vspace*{0.5cm}

\noindent\textbf{Acknowledgement:}
A. Chattopadhyay is supported by the Core Research Grant (CRG), File No: CRG/2023/004826, from SERB. S. Giri is grateful to the Ministry of Education, Government of India, for financial support through the Prime Minister's Research Fellowship (PMRF-ID: 1902164) grant. S. Jain is supported by SERB project (Ref No. MATHSPNSERB01119xARC002).

\end{document}